\documentclass[12pt]{amsart}
\usepackage{amssymb,latexsym,amsthm, amsmath, comment, fullpage, mathtools, tikz-cd}
\usepackage[all]{xy}
\usepackage{graphicx, comment, libertine}
\usepackage{xcolor}
\usepackage[normalem]{ulem}

\ifx\pdfoutput\undefined
\usepackage{hyperref}
\else
\usepackage[pdftex,colorlinks=true,linkcolor=blue,urlcolor=blue]{hyperref}
\fi

\theoremstyle{plain}
\newtheorem{theorem}{Theorem}
\newtheorem{corollary}{Corollary}
\newtheorem{proposition}{Proposition}

\newtheorem{lemma}{Lemma}
\theoremstyle{definition}
\newtheorem{definition}{Definition}
\newtheorem{remark}{Remark}

\newtheorem*{ack}{Acknowledgements}

\everymath{\displaystyle}

\begin{document}

\title[Random substitution tilings and deviation phenomena]
      {Random substitution tilings and deviation phenomena}
\author[S. Schmieding]{Scott Schmieding}
\address{Northwestern University}
\email{schmiedi@math.northwestern.edu}
\author[R. Trevi\~{n}o]{Rodrigo Trevi\~no}
\address{University of Maryland}
\email{rodrigo@math.umd.edu}

\begin{abstract}
  Suppose a set of prototiles allows $N$ different substitution rules. In this paper we study tilings of $\mathbb{R}^d$ constructed from random application of the substitution rules. The space of all possible tilings obtained from all possible combinations of these substitutions is the union of all possible tilings spaces coming from these substitutions and has the structure of a Cantor set. The renormalization cocycle on the cohomology bundle over this space determines the statistical properties of the tilings through its Lyapunov spectrum by controlling the deviation of ergodic averages of the $\mathbb{R}^d$ action on the tiling spaces.
\end{abstract}
\maketitle

\section{Introduction}
In this paper we study tilings which are generated by random combinations of substitutions using a finite family of substitution rules. This generalizes the constructions and results known for self-similar tilings, which are tilings constructed from a single substitution rule. As an example to keep in mind, consider the two substitution rules defined for the following triangles:
\begin{figure}[h]
  \centering
  \includegraphics{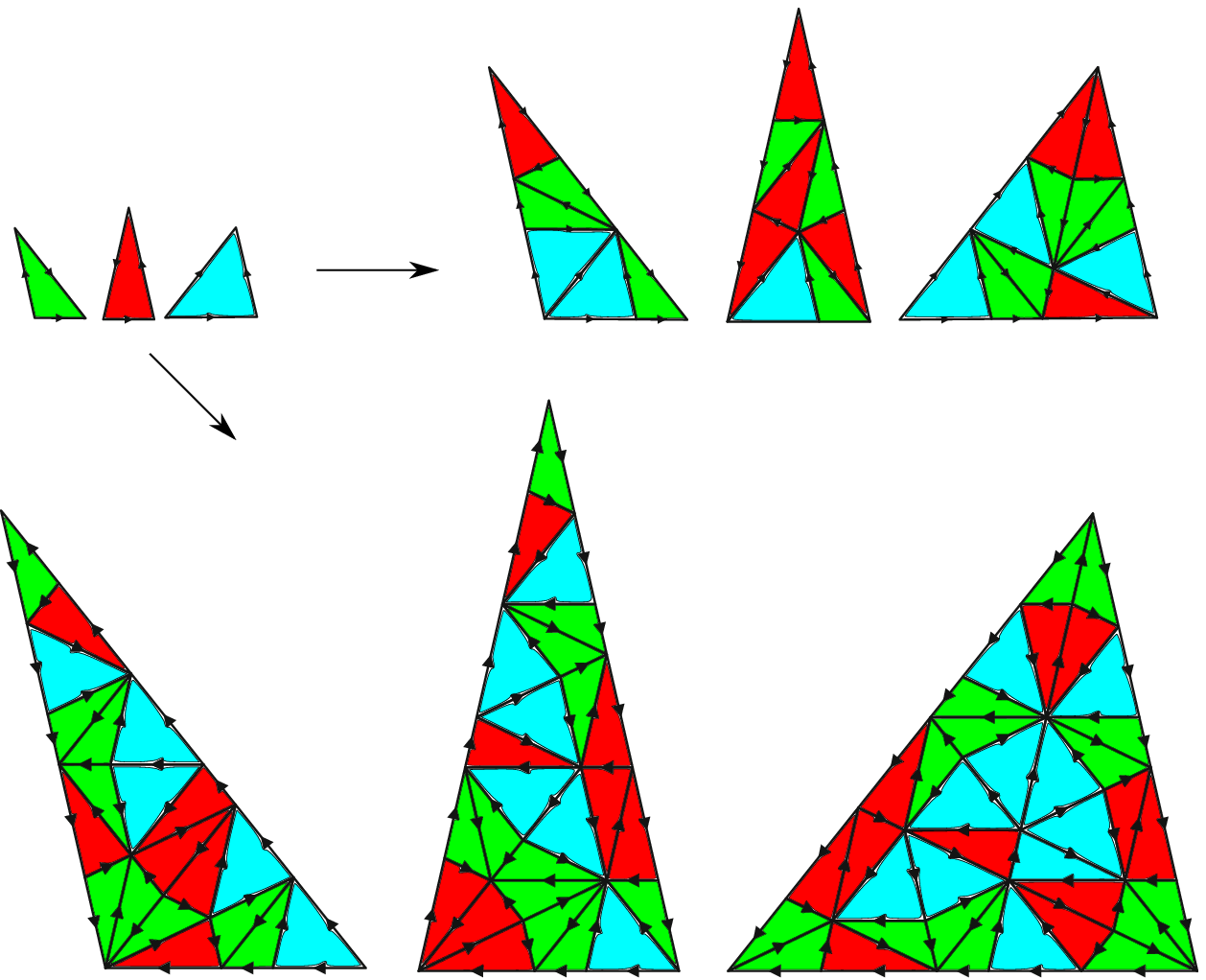}
\end{figure}

We first point out that the two substitutions are in fact different; it is not the case that one is the power of another. Moreover, their expansion constants are not related by a power. These two substitution rules were discovered in \cite{GKM:computer}.

The general procedure for constructing self-similar tilings from a single substitution rule can be roughly described as follows: start with a single tile, apply the substitution rule and rescale the tiled polygon so that the tiles in the polygon are isometric copies of the tiles on which the substitution rule is defined. Doing this infinitely many times and carefully taking a limit, one obtains a self-similar tiling.

Now, for the triangles in the figure above, suppose that instead of using a single substitution rule to build a tiling one applies a sequence of substitutions randomly chosen from among the two substitutions given above to construct a tiling. Suppose $\mathcal{T}_x$ and $\mathcal{T}_{x'}$ are two different tilings constructed from two sequences $x\neq x'\in \{1,2\}^\mathbb{N}$, the entries of which determine the order in which we apply either substitution rule. We may ask:
\begin{enumerate}
\item How are $\mathcal{T}_x$ and $\mathcal{T}_{x'}$ related?
\item How are the respective tiling spaces of $\mathcal{T}_x$ and $\mathcal{T}_{x'}$ related?
\item Will the dynamics defined by $\mathcal{T}_x$ and $\mathcal{T}_{x'}$ be conjugate?
\item What determines the statistical properties of the two tilings $\mathcal{T}_x$ and $\mathcal{T}_{x'}$ such as asymptotic patch frequency?
\end{enumerate}
The answers to these questions in the self-similar case are well-known to be related to the geometry and combinatorics of the substitution rule. In this paper we show that what determines the answers to these and other questions are the ergodic shift-invariant measures on $\{1,2\}^\mathbb{N}$, the typical points $x,x'$ of which which we take to construct tilings.

Our construction of tilings using graph iterated function systems is inspired by the blowup construction of Barnsley and Vince \cite{BarnsleyVince:blowup} but we make use of Bratteli diagrams to organize and give structure to all of the possible combinations of substitutions we may use. The use of Bratteli diagrams in the study of tilings goes back several decades, see e.g. \cite{kellendonk:gap, BJS:brat1, JS:bratII}. Our use of Bratteli diagrams can particularly be seen as a non-stationary version of those used in \cite{kellendonk:gap}. Our formalism using Bratteli diagrams also has many parallels to the fusion theory of Priebe-Frank and Sadun \cite{FrankSadun:Fusion}.

There have been other works where random substitutions have been investigated \cite{GM:mixed, BD:sadic, Rust:uncountable, RS:random} but most of the results in those are one-dimensional in nature. A difficulty which arises in the case of higher dimensional tilings, which is of independent interest in itself, is whether a given set of tiles admits more than one substitution rule. The results of \cite{GKM:computer} indicate that although this is a hard question in general, one can find plenty of interesting examples by considering triangles with angles which are integer multiples of $\pi/n$ for most $n>4$. The figure above is one of many examples found in \cite{GKM:computer}.

Our approach here is the one adopted in the study of translation flows in Teichm\"uller dynamics. To summarize, given a finite set of substitution rules we can consider all possible tiling spaces which can be constructed from these subsitution rules. This serves as a sort of ``moduli space'' of tiling spaces coming from a given family of substitutions on the same set of tiles. There is a dynamical system on this moduli space and the dynamics on this moduli space determine many of the properties of the tilings constructed. The dynamics on the moduli space are known as \textbf{renormalization dynamics}.

Tiling spaces associated to aperiodic tilings are foliated spaces which are not manifolds; instead, they are locally the product of a manifold with a Cantor set. As such, in contrast to the situation in Teichm\"uller dynamics, there is no Hodge theory for tiling spaces, so we have to come up with some components which are missing in the context of tiling spaces, such as a useful norm on the relevant cohomology bundle. Moreover, it is not clear whether there are Sobolev spaces where de Rham regularization yields an isomorphism between finite-dimensional smooth cohomology and any type of finite-dimensional Sobolev cohomology, so analytic approaches using Sobolev norms (e.g. \cite{forni:deviation}, \cite{FF:horocycle}, \cite{CF:equidistr}) are not clearly applicable in this setting.

A thorough study of the moduli spaces defined by families of substitutions falls outside the scope of the present paper, and the investigation of such moduli spaces is a topic we plan pursue in future work. Thus, while we do not explicitly call anything in the paper an actual moduli space, the reader familiar with dynamics on moduli spaces will recognize our use of the shift on $\Sigma_N$ as the dynamics on moduli space where the action of some mapping class group of a tiling space acts through a map induced by the shift.

Our blowup construction through graph iterated function systems to construct tilings and tiling spaces is quite general. The restrictions we impose here allow us to obtain tilings and tiling spaces which have finite complexity (and, in addition, finite-dimensional cohomology). However, relaxing these restrictions may give tilings of several, even infinite, scales, as well as tilings of infinite complexity. Our hope is to extend the renormalization tools used in this paper to study the more general case of multiscale and tilings of infinite complexity.

As mentioned above, the constructions here generalize the construction of self-similar tilings, and our main result generalizes the results \cite{sadun:exact, BufetovSolomyak, ST:SA} to the context of not-self similar tilings. These types of results are not only illuminating in the study of tilings, but also are of interest to the mathematical physics community. Since mathematical tilings are taken as models for quasicrystals, the results here yield results about the convergence properties of diffraction measures for quasicrystals (see \cite{ST:SA}). In addition, these results also yield information about convergence properties in the Bellissard-Shubin formula for the integrated density of states for random Schrodinger operators on quasiperiodic media, as well as traces in the $*$-algebras of certain types of operators as in \cite{ST:traces}.


\subsection{Statement of results}
Suppose we have $N$ substitution rules $\mathcal{F}_1,\dots , \mathcal{F}_N$ on the same set of prototiles which satisfy certain conditions (see Definition \ref{def:typeH} in \S\ref{sec:blowups}). The assumptions guarantee that most tilings constructed from these substitution rules will have finite complexity. Given $x\in \{1,\dots, N\}^{\mathbb{Z}}$, we can construct a (bi-infinite) Bratteli diagram $\mathcal{B}_x$ which records a set of instructions used to create a tiling. A Bratteli diagram is an infinite directed graph partitioned into levels indexed by $\mathbb{Z}$ (Bratteli diagrams are defined in \S \ref{sec:gifs}), so the $k^{th}$ level of $\mathcal{B}_x$ is defined by $x_{k}$. We construct tilings from infinite paths in $\mathcal{B}_x$, and as long as $\mathcal{B}_x$ is connected enough, the collection of all such tilings gives a tiling space $\Omega_x$ with an action of $\mathbb{R}^d$ given by translations. We call such sufficiently connected diagrams $\mathcal{B}_x$ \emph{minimal} (minimal diagrams are defined in \S\ref{subsec:brat}), and minimal diagrams yield tiling spaces with minimal $\mathbb{R}^d$-actions. A shift-invariant measure on $\Sigma_N$ is minimal if $\mathcal{B}_x$ is minimal for $\mu$-almost every $x$.

The shift $\sigma:\Sigma_N\rightarrow\Sigma_N$ defines a homeomorphism $\Phi_x:\Omega_x\rightarrow \Omega_{\sigma(x)}$ which is a conjugacy between the translation actions in $\Omega_x$ and $\Omega_{\sigma(x)}$, respectively, and drives the renormalization dynamics (this is found in \S\ref{sec:biinf}). We define the cohomology bundle $\mathcal{H}_\mathcal{F}$ over $\Sigma_N$ where the fiber over $x$ is the vector space $H^d(\Omega_x;\mathbb{R})$ and the renormalization cocycle is the bundle map $(x,c)\mapsto (\sigma(x), (\Phi_x^{-1})^* c)$ over the shift $\sigma$.

As such, given a $\sigma$-invariant ergodic minimal measure $\mu$ {on $\Sigma_{N}$}, Oseledets theorem yields Lyapunov exponents $\lambda_1\geq \dots\geq \lambda_r$ which measure the exponential rate of growth of vectors in $H^d(\Omega_x;\mathbb{R})$ under the renormalization cocycle. The \emph{rapidly expanding subspace} $E^+_x\subset H^d(\Omega_x;\mathbb{R})$ corresponds to vectors with Lyapunov exponents $\lambda_i$ satisfying $d\lambda_i > (d-1)\lambda_1$. The functions whose ergodic integrals we study are the analogue of $C^\infty$ functions on manifolds, which are the \emph{transversally locally constant functions}, denoted $C^\infty_{tlc}(\Omega_x)$ (they are defined in \S\ref{subsec:AP}). For a set $B\subset \mathbb{R}^d$ we denote by $T\cdot B$ the rescaling of $B$ by $T>0$, that is, $T\cdot B = T\,\mathrm{Id}\,B$.
\begin{theorem}
\label{thm:main}
  Let $\mathcal{F} = \{\mathcal{F}_1,\dots, \mathcal{F}_N\}$ be a family of substitution rules satisfying the conditions of Definition \ref{def:typeH} in \S\ref{sec:biinf}. Let $\mu$ be a minimal $\sigma$-invariant ergodic probability measure on $\Sigma_N$, and let $\lambda_1\geq \cdots\geq  \lambda_{\rho}$ be the Lyapunov exponents for $\mu$ corresponding to vectors in $E^+_x$. Then for $\mu$-almost every $x\in\Sigma_N$, there are $\rho$ $\mathbb{R}^d$-invariant distributions $\mathcal{D}_1,\dots, \mathcal{D}_{\rho} \in C^\infty_{tlc}(\Omega_x)'$ such that for any $f\in C^\infty_{tlc}(\Omega_x)$, if $\mathcal{D}_i(f)=0$ for all $i<j\leq \rho$ and $\mathcal{D}_j(f)\neq 0$, for a good Lipschitz domain $B\subset \mathbb{R}^d$ and $\mathcal{T}\in\Omega_x$ we have that
  $$\limsup_{T\rightarrow\infty} \frac{\log \left| \displaystyle\int_{T\cdot B}f\circ\varphi_t(\mathcal{T})\, dt\right|}{\log T} \leq d\frac{\lambda_j}{\lambda_1}.$$
Moreover, for any $\varepsilon>0$ there exists a compact subset $B_\varepsilon$ which is $\varepsilon$-close in the Hausdorff metric to $B$, a convergent sequence of vectors $\tau_k\in \mathbb{R}^d$ and a sequence $T_k\rightarrow \infty$ such that
  $$\limsup_{k\rightarrow\infty} \frac{\log \left| \displaystyle\int_{T_k\cdot(\tau_k+B_\varepsilon)}f\circ\varphi_t(\mathcal{T})\, dt\right|}{\log T_k} \geq d\frac{\lambda_j}{\lambda_1}.$$
  Finally, if $\mathcal{D}_i(f)=0$ for all $i\leq \rho$, then
  $$\limsup_{T\rightarrow\infty} \frac{\log \left| \displaystyle\int_{T\cdot B}f\circ\varphi_t(\mathcal{T})\, dt\right|}{\log T} \leq d-1.$$
\end{theorem}
\begin{remark}
As in the case of translation flows \cite{forni:deviation}, the lower bound is harder to obtain than the upper bound, and the geometry of the group acting on the spaces comes into play in the derivation of a lower bound. For tilings of dimension greater than 1 ($d>1$), unlike the case of flows, the geometry of $\mathbb{R}^d$ is nontrivial, which is why we must make small changes to the averaging sets to obtain a lower bound along a subsequence.
\end{remark}
\subsection{Outline}
This paper is organized as follows. In \S\ref{sec:back} we review the necessary materials for tilings and tilings spaces. In \S\ref{sec:gifs} we review graph iterated function systems as well as Bratteli diagrams and construct Bratteli digrams from graph iterated functions systems. In \S\ref{sec:blowups} we show how to construct tilings using infinite paths on the Bratteli diagrams constructed from families of graph iterated function systems. We also relate the strcuture of the Bratteli diagram to the structure of the tiling space. In \S\ref{sec:biinf} we extend the construction to bi-infinite Bratteli diagrams and introduce the renormalization operations on the tiling spaces. \S\ref{sec:AP} concerns the cohomology of the tiling spaces constructed and it culminates with explicit norm on the cohomology spaces of top degree for tiling spaces. In \S\ref{sec:bundle} we define the cohomology bundle and define the renormalization cocycle. Using all this, in \S \ref{sec:erg}, we prove the main results on deviation of ergodic averages. The route we follow is inspired by Forni's work on translation surfaces\cite{forni:deviation} (see also \cite[\S 5]{DHL:WindTree}).
\begin{ack}
We would like to thank Giovanni Forni for pointing out a gap in the paper in an early draft. S.S was supported in part by the National Science Foundation grant `RTG: Analysis on manifolds' at Northwestern University. R.T was supported by the National Science Foundation through grant DMS-1665100.
\end{ack}

\section{Background}
\label{sec:back}
A \textbf{tile} $t$ is a bounded, connected subset of $\mathbb{R}^d$. We assume tiles have non-empty interior and regular boundary. A \textbf{tiling $\mathcal{T}$ of $\mathbb{R}^d$ by tiles $\{t_i\}_i$} is a cover of $\mathbb{R}^d$ by translated copies of the tiles $t_i$ such that any two different tiles in this cover intersect, at most, along their boundaries. Here we are only concerned with tilings obtained using copies of a finite set of tiles $\{t_1,\dots, t_M\}$, called the set of \textbf{prototiles}. A \textbf{patch} $\mathcal{P}$ of the tiling $\mathcal{T}$ is a finite subset of the tiles of $\mathcal{T}$, and the support of a patch $\mathcal{P}$ is the union of the tiles contained in $\mathcal{P}$. Finally, denote by $\partial \mathcal{T}$ the union of the boundaries of all the tiles covering $\mathbb{R}^d$ in the tiling $\mathcal{T}$, and $\partial \mathcal{P}$ the union of boundaries of the tiles contained in the patch $\mathcal{P}$ of $\mathcal{T}$. We say a tiling $\mathcal{T}$ is \textbf{regular} if the set $\partial\mathcal{T}$ is closed in $\mathbb{R}^d$. In this paper we will only consider regular tilings.

A tiling $\mathcal{T}$ admits a \textbf{substitution rule} if there exists a scaling factor $s\in (0,1)$ such that each prototile $t_i$ can be tiled by the prototiles $\{st_1,\dots ,st_M\}$. A tiling which admits a substitution rule is called a \textbf{substitution tiling}.

Tilings can be pushed around: for any $\tau\in\mathbb{R}^d$ we denote by $\varphi_\tau(\mathcal{T}) = \mathcal{T}+\tau$ the translation of the tiling $\mathcal{T}$ by the vector $\tau$. A tiling $\mathcal{T}$ is \textbf{repetitive} if for any patch $\mathcal{P}\subset \mathcal{T}$ there exists an $R>0$ such that for any $x\in\mathbb{R}^d$ the set $B_x(R)\cap\mathcal{T}$ contains a translated copy of $\mathcal{P}$. A tiling $\mathcal{T}$ has \textbf{finite local complexity} if for every $R>0$ there exists a set of patches $\mathcal{P}_1^R,\dots, \mathcal{P}_{N_R}^R$ such that for any $x\in\mathbb{R}^d$ the union of all the tiles of $\mathcal{T}$ which intersect $B_x(R)$ is a translated copy of one of the patches $\mathcal{P}_i^R$. A tiling $\mathcal{T}$ is \textbf{aperiodic} if $\varphi_\tau(\mathcal{T}) = \mathcal{T}$ implies that $\tau = 0$. In this paper we will only be concerned with aperiodic tilings of finite local complexity.

Denote by $\Pi_d:\mathbb{R}^d\rightarrow S^d$ the inverse of the stereographic projection. We can impose a distance on the set of all translates $\varphi_\tau(\mathcal{T})$ of a regular tiling $\mathcal{T}$ by
\begin{equation}
  \label{eqn:distance}
d(\mathcal{T},\varphi_\tau(\mathcal{T})) = d_H(\Pi_d(\partial\mathcal{T}),\Pi_d(\varphi_\tau(\partial\mathcal{T}))),
\end{equation}
where $d_H(X,Y)$ is the Hausdorff distance of two closed subsets $X,Y\subset S^d$. The completion
\begin{equation}
  \label{eqn:space}
  \Omega_\mathcal{T} := \overline{\{\varphi_\tau(\mathcal{T}):\tau\in\mathbb{R}^d\}}
\end{equation}
with respect to the metric (\ref{eqn:distance}) is called the \textbf{tiling space} of $\mathcal{T}$. As such, at admits an action of $\mathbb{R}^d$ by translation and thus is foliated by the orbits of this action. It is compact if $\mathcal{T}$ has finite local complexity, and the translation action of $\mathbb{R}^d$ is minimal if and only if $\mathcal{T}$ is repetitive.

Let $\mathcal{T}$ be a regular, repetitive tiling of $\mathbb{R}^d$ of finite local complexity whose tiles are all copies of a finite set of prototiles $\{t_1,\dots, t_M\}$. Pick a point $p_i \in t_i$ in the interior of each prototile. Then each tile in the tiling $\mathcal{T}$ has a distinguished point in its interior coming from the distinguished points $p_i$. The \textbf{canonical transversal}
\begin{equation}
  \label{eqn:mho}
\mho_\mathcal{T} = \{\mathcal{T}^{\prime
} \in \Omega_\mathcal{T} \mid \textnormal{ the origin is the distinguished point of the tile in } \mathcal{T}^{\prime} \textnormal{ containing the origin}\}
\end{equation}
is a Cantor subset of $\Omega_\mathcal{T}$ if $\mathcal{T}$ has finite local complexity. Its name comes from the fact that it intersects every $\mathbb{R}^d$ orbit. This set depends on our choice of distinguished points for the prototiles, but we get homeomorphic sets as long as our choice for distinguished points in every tile is uniform. The following is well known.

\begin{proposition}
  \label{prop:productStruct}
Let $\mathcal{T}$ be an aperiodic, repetitive tiling of finite local complexity. The topological space $\Omega_{\mathcal{T}}$ has a basis given by sets of the form $\mathcal{C}\times V$, where $\mathcal{C}$ is a Cantor set and $V\subset \mathbb{R}^d$ is homeomorphic to an open disk.
  \end{proposition}
For a closed subset $S\subset \mathbb{R}^d$, a tiling $\mathcal{T}$, and $r>0$, define the sets
\begin{equation*}
  \begin{split}
    \mathcal{O}^+_\mathcal{T}(S) &= \mbox{union of all tiles in $\mathcal{T}$ intersecting }S,\\
    \mathcal{O}^-_\mathcal{T}(S) &= \mbox{union of all tiles in $\mathcal{T}$ completely contained in }S, \\
    \partial_r(S) &= \mbox{$r$-neighborhood of the boundary $\partial S$ of }S.
  \end{split}
\end{equation*}

\subsection{Lipschitz domains}
Let $\mathcal{H}^m$ denote the $m$-dimensional Hausdorff measure.
\begin{definition}
A set $E\subset \mathbb{R}^d$ is called \emph{$m$-rectifiable} if there exist Lipschitz maps $f_i: \mathbb{R}^m\rightarrow\mathbb{R}^d$, $i = 1,2,\dots$ such that
$$\mathcal{H}^m\left(  E\backslash \bigcup_{i\geq 0} f_i(\mathbb{R}^m)   \right) = 0.$$
\end{definition}
\begin{definition}
A \emph{Lipschitz domain} $A\subset\mathbb{R}^d$ is an open, bounded subset of $\mathbb{R}^d$ for which there exist finitely many Lipschitz maps $f_i:\mathbb{R}^{d-1}\rightarrow \mathbb{R}^d$, $i = 1,\dots, L$ such that
$$\mathcal{H}^{d-1}\left(  \partial A \backslash \bigcup_{i=1}^L f_i(\mathbb{R}^{d-1})   \right) = 0.$$
\end{definition}
Lipschitz domains have $d-1$-rectifiable boundaries.
\begin{definition}
A subset $A\subset \mathbb{R}^d$ is a \textbf{good} Lipschitz domain if it is a Lipschitz domain and $\mathcal{H}^{d-1}(\partial A)<\infty$.
\end{definition}

\section{Graph iterated function systems}
\label{sec:gifs}
Here we recall the basics of graph iterated function systems (GIFS), our goal being to build a graph which will represent an iterated function system. Suppose we have $M\in\mathbb{N}$ copies of $\mathbb{R}^d$, denoted by $\mathbb{R}_1^d,\dots, \mathbb{R}^d_M$, and let
$$X =  \mathbb{R}^d_1\times \cdots \times \mathbb{R}^d_M.$$
Suppose we have $r(i,j)\in\mathbb{N}$ maps $f_{i,j,k}:\mathbb{R}^{d}_i\rightarrow \mathbb{R}^{d}_j$, with $k\in\{1,\dots, r(i,j)\}$. Suppose $S\subset X$ is of the form $S = S_1\times\cdots\times S_{M}$, where $S_i\subset \mathbb{R}^{d}_i$. The GIFS associated to this collection of maps is the mapping of sets defined as
\begin{equation}
  \label{eqn:bigMap}
F(S) =\left(\bigcup_{i=1}^{M}\bigcup_{k=1}^{r(i,1)} f_{i,1,k}(S_i),\dots ,\bigcup_{i=1}^{M}\bigcup_{k=1}^{r(i,M)} f_{i,M,k}(S_i)\right).
  \end{equation}
An \textbf{attractor} for the GIFS $F$ is a set $A = A_1\times\cdots\times A_{M} \subset X$ satisfying $F(A) = A$.

The following is a more general point of view. Let $\mathcal{C}$ be the set of all closed subsets of $\mathbb{R}^d$ endowed with topology induced by the Hausdorff metric, which makes it a compact metric space. Let $\mathcal{C}_k = \mathcal{C}\times\cdots \times \mathcal{C}$ be the Cartesian product of $\mathcal{C}$ with itself $k$ times with the product topology. A GIFS $F$ as above induces a map $\mathcal{F}:\mathcal{C}_{M}\rightarrow\mathcal{C}_{M}$ as follows. Let $\mathcal{S} = (\mathcal{S}_1,\dots,\mathcal{S}_{M})\in \mathcal{C}_{M}$. Then
$$\mathcal{F}(\mathcal{S}) =\left(\bigcup_{i=1}^{M}\bigcup_{k=1}^{r(i,1)} f_{i,1,k}(\mathcal{S}_i),\dots ,\bigcup_{i=1}^{M}\bigcup_{k=1}^{r(i,M)} f_{i,M,k}(\mathcal{S}_i)\right).$$
It is well known that if each $f_{i,j,k}$ is a contraction, then $\mathcal{F}$ is a contraction. As such, by Hutchinson's theorem \cite{Hutchinson}, there is a fixed point for $\mathcal{F}$ which is an attractor for $F$. A GIFS $F$ is \textbf{contracting, uniform affine scaling} (CUAS) if there exists a $s\in (0,1)$ such that all maps are of the form $f(x) = sx + q$, for some $q\in\mathbb{R}^d$.
\begin{lemma}
  \label{lem:subrule}
Any substitution rule is given by a CUAS GIFS.
\end{lemma}
\begin{proof}
For $\mathcal{T}$ to be a substitution tiling it needs to admit a substitution rule. By definition, a substitution rule gives a way of covering each prototile $t_i$ with copies of scaled prototiles $\{st_1,\dots, st_N\}$. So for our GIFS we take $M$ to be the number of prototiles and the maps $f_{i,j,k}$ the different maps which take each prototile into another prototile. Since it is a substitution tiling, the attractor is the product of the prototiles.
\end{proof}
Given a GIFS $\mathcal{F} = \{f_{i,j,k}\}$ we can associate a graph as follows. The graph will have $|V| = M$ vertices labeled $v_1,\dots, v_{|V|}$ and there will be $r(i,j)$ directed edges going from vertex $v_i$ to vertex $v_j$. Note that there is a bijection between the edges of the graph and the maps $f_{i,j,k}$ of the GIFS.
\begin{definition}
A set of GIFS $\mathcal{F} = \{\mathcal{F}_1,\dots, \mathcal{F}_N\}$ is said to have a \textbf{shared attractor} if the attractor for $\mathcal{F}_i$ is the same as the attractor for $\mathcal{F}_j$ for all $i,j$.
\end{definition}

\begin{definition}
  \label{def:CUASfam}
  A family of GIFS $\mathcal{F} = \{\mathcal{F}_1,\dots, \mathcal{F}_N\}$ is called \textbf{contracting, uniformly affine scaling} (CUAS) if there are $(\theta_1,\dots,\theta_N)\in(0,1)^N$ such that all maps associated to $\mathcal{F}_i$ are of the form $f(x) = \theta_i x + b$ for some $b\in\mathbb{R}^d$.
\end{definition}
\begin{remark}
\label{rem:convention}
  Given a CUAS family $\mathcal{F}$ of GIFS with shared attractor $A$, without loss of generality, \textbf{we will always assume that the origin is contained in the interior of the attractor}. Whenever the attractor corresponds to the product of the prototiles in a substitution tiling this can be done by choosing a distinguished point in the interior of each prototile and making this distinguished point the origin.
  \end{remark}
\subsection{Bratteli diagrams and GIFS}
\label{subsec:brat}
A \textbf{Bratteli diagram} is an infinite directed graph $B  = (V,E)$ with both the vertex and edges sets partitioned as
$$V = \bigsqcup_{k\geq 0}V_k\hspace{1in}\mbox{ and } \hspace{1in}E = \bigsqcup_{k>0}E_k$$
with surjective maps $r:E_k\rightarrow V_{k}$ and $s:E_k\rightarrow V_{k-1}$ called, respectively, the \textbf{range and source maps}. Since the graph is directed, the maps $r,s$ describe where individual edges end and begin, respectively. We shall always assume that the sets $|V_k|$ and $|E_k|$ are finite for all $k$.

A Bratteli diagram can also be described by the transitions between levels. That is, the data of the edges between $V_{k-1}$ and $V_{k}$ is given by a matrix $M_k$ defined by
$$(M_k)_{i,j} = \textnormal{ number of edges between } v_i \in V_k \textnormal{ and } v_j \in V_{k-1}.$$
The matrix $M_{k}$ is called the \textbf{$k^{th}$ transition matrix} of the Bratteli diagram.

A finite path of a Bratteli diagram is a collection of edges $\bar{e} = (e_i,\dots, e_j)$ with $r(e_k) = s(e_{k+1})$ for all $k = 1,\dots, j-1$. We extend the domain of the source map to the set of all finite paths by assigning the source of a path to be the same vertex which is the source of the first edge of the path. Likewise, we can extend the domain of the range map to all finite paths by assigning the range of the last edge on the path. We denote by $E_{p,q}$ the set of all paths with source in $V_p$ and range in $V_q$. For $v\in V_k$, we also denote by $E_v$ the set of all paths $\bar{e}$ with $s(\bar{e})\in V_0$ and $r(\bar{e}) = v$, i.e. all paths which end in the vertex $v$.

An infinite path of a Bratteli diagram is a collection of edges $\bar{e} = (e_i,e_{i+1},\dots )$ with $r(e_k) = s(e_{k+1})$ for all $k >{i-1}$ and we extend the domain of the source map to include infinite paths in the obvious way. We denote the set of all infinite paths with source in $V_0$ by $X_B$ and endow it with the (infinite) product topology coming from the fact that $X_B$ can be seen as a subset of the infinite product of sets of edges $E_k$. Given $\bar{e} = (e_1,e_2,\dots)\in X_B$ we define $\bar{e}|_k$ to be the finite path $(e_1,\dots, e_k)$.

The topology of $X_B$ is generated by \textbf{cylinder sets}: if $\bar{e}|_k$ is a finite path, we define $C_{\bar{e}|_k}$ to be the open set of all paths which agree with $\bar{e}$ in the first $k$ edges. The collection of such cylinder sets $C_{\bar{e}}$ forms a basis for the topology on $X_{B}$, and $X_B$ is a compact totally disconnected space.

The \textbf{tail of a path $\bar{e}\in X_B$ from level $k$} is the infinite path $(e_{k+1},e_{k+2},\dots)$. Two paths $\bar{e}$ and $\bar{f}$ are \textbf{tail equivalent} if there exists a $k$ so that the tail of $\bar{e}$ from level $k$ is the same as the tail of $\bar{f}$ from level $k$. This is an equivalence relation on $X_{B}$ and we denote by $[\bar{e}]$ the tail-equivalence class of $\bar{e}\in X_B$. A Bratteli diagram is \textbf{minimal} if for any $\bar{e} \in X_{B}$ the tail-equivalence class $[\bar{e}]$ is dense in $X_B$. A tail-equivalence class $[\bar{e}]$ is called a \textbf{periodic component} of $X_B$ if it is finite.

\begin{definition}
A Borel probability measure $\mu$ on $X_{B}$ is \textbf{invariant under the tail equivalence relation} if for any two finite paths $\bar{e}$, $\bar{e}'$ with the property that $s(\bar{e}),s(\bar{e}')\in V_0$ and $r(\bar{e})=r(\bar{e}')$, we have that $\mu(C_{\bar{e}})=\mu(C_{\bar{e}'})$.
\end{definition}
Such measures will be referred to as invariant measures. If $\mu$ is an invariant measure on $X_B$ and $v\in V_k$, then we define
\begin{equation}
  \label{eqn:invMeasure}
\mu(v):= \mu(C_{\bar{e}}),
  \end{equation}
for any $\bar{e}\in E_{0,k}$ with $r(\bar{e}) = v$. By definition of invariance, this is independent of the path chosen in $E_{0,k}$.
\subsubsection{Bratteli diagrams and GIFS}
\label{subsubsec:Bratgifs}
Suppose that we have $N$ substitution rules defined on the same set of prototiles. By Lemma \ref{lem:subrule}, these are given by a CUAS family of GIFS $\mathcal{F} = \{\mathcal{F}_1,\dots, \mathcal{F}_N\}$ with shared attractor. Each GIFS $\mathcal{F}_i$ defines a matrix $\mathcal{M}_k=\mathcal{M}(\mathcal{F}_k)$ with integer entries: $\mathcal{M}(\mathcal{F}_k)_{i,j}$ is the integer $r(i,j)$ coming from the GIFS in (\ref{eqn:bigMap}). Denote by $\mathcal{M}_1,\dots, \mathcal{M}_N$ the different matrices for $\mathcal{F}$ and define $\bar{Z}:= \mathbb{Z}-\{0\}$. For $x = (x_1,x_2,\dots)\in \Sigma_N := \{1,\dots, N\}^{\bar{\mathbb{Z}}}$, the Bratteli diagram $B_x(\mathcal{F}) = B_x(\mathcal{F}_1,\dots, \mathcal{F}_N) = (V_x,E_x)$ is the Bratteli diagram with transition matrix $\mathcal{M}_k = \mathcal{M}(\mathcal{F}_{x_k})$ between $V_{k-1}$ and $V_{k}$ for all $k>0$. This is called \textbf{Bratteli diagram with parameter $x$}. Note that in this construction there is a map $f_{i,j,k}$ associated to each edge $e\in E_x$.
  \begin{remark}
Note that even though the Bratteli diagram $B_x(\mathcal{F})$ only depends on the coordinates of $x\in\Sigma_N$ with index greater than zero, we still take $x$ to be a bi-infinite sequence and not only an infinite sequence. This is because having an infinite past will help us define homeomorphisms between tiling spaces. This will become clear in \S \ref{sec:renorm}.
  \end{remark}

\section{Blowups and random substitutions}
\label{sec:blowups}
The following condition first appeared in \cite{GM:mixed} and it ensures that a family $\mathcal{F}$ of GIFS gives a substitution rule with enough structure to guarantee finite complexity.
\begin{definition}[Compatibility]
  A family $\mathcal{F} = \{\mathcal{F}_1,\dots,\mathcal{F}_N\}$ of GIFS with shared attractor $A = A_1\times\cdots\times A_M$ are \textbf{compatible} if for every $i$, $A_i$ has a $CW$-structure and if for any $v\in V-V_0$, for any $\bar{e},\bar{e}'\in E_v$ with $f_{\bar{e}}(A_{s(\bar{e})})\cap f_{\bar{e}'}(A_{s(\bar{e}')})\neq \varnothing$, the intersection is a union of $d-1$ cells in both $f_{\bar{e}}(A_{s(\bar{e})})$ and $f_{\bar{e}'}(A_{s(\bar{e}')})$.
\end{definition}
In order to reduce the tedious number of adjectives assigned to families of graph iterated function systems we make the following definition.
\begin{definition}[Type H]
  \label{def:typeH}
  A \textbf{type H family} $\mathcal{F}$ is a finite collection $\{\mathcal{F}_1,\dots, \mathcal{F}_N\}$ of graph iterated function systems which
  \begin{enumerate}
  \item is contracting,
  \item is uniformly affine scaling,
  \item has a shared attractor containing the origin,
  \item is compatible.
  \end{enumerate}
\end{definition}
Let $\mathcal{F} = \{\mathcal{F}_1,\dots, \mathcal{F}_N\}$ be a type H family and pick $x\in \Sigma_N$. Let $B_x(\mathcal{F}) = (V_x(\mathcal{F}), E_x(\mathcal{F}))$ be the Bratteli diagram given by the family $\mathcal{F}$ GIFS and parameter $x$. Note that the number of vertices is the same for all levels (as it is given by the number of prototiles in each of the substitutions) and we denote this number by $M=|V|$. Recall that the set of edges is in bijection with contracting maps $f_{i,j,k}$ of $\mathbb{R}^d$ in (\ref{eqn:bigMap}). Thus, to any edge $e\in E_x(\mathcal{F})$ there is a unique contracting map $f_e:\mathbb{R}^d\rightarrow \mathbb{R}^d$.

Given a finite path $\bar{e} = (e_{p+1},\dots,e_q)\in E_{p,q}$ on $B_x(\mathcal{F}_1,\dots, \mathcal{F}_N)$ we define an associated map
\begin{equation}
  \label{eqn:comp}
f_{\bar{e}} := f_{e_{q}}\circ f_{e_{q-1}}\circ \cdots \circ f_{e_{p+1}}:\mathbb{R}^d\rightarrow \mathbb{R}^d.
\end{equation}
Starting from the attractor $A = A_1\times \cdots \times A_{M} $ we can build a sequence of tiled patches of arbitrarily large size through ``blowups'' \cite{BarnsleyVince:blowup}.

Pick $\bar{e}\in E_{0,k}$. The idea behind blowups is first to consider each part $A_i$ of the attractor $A$ as a prototile, and label the prototiles with the vertices $\{v_1,\dots, v_{M}\}$. Since finite paths $\bar{e}$ give us maps through composition as in (\ref{eqn:comp}), we can start with a part of the attractor $A_i$ and apply the inverse of the map $f_{\bar{e}}$ to ``blow up'' $A_i$. More precisely, for $\bar{e}\in E_{0,k}$ consider the set $f_{\bar{e}}^{-1}(A_{r(\bar{e})})$, where $A_{r(\bar{e})} = A_j$  if $r(\bar{e}) = v_j\in V_k$.
This set $f_{\bar{e}}^{-1}(A_{r(\bar{e})})$ is the rescaling of $A_{r(\bar{e})}$ by the factor $\theta_{x_1}^{-1}\cdots \theta_{x_k}^{-1}$, where the $\theta_i$'s are the scaling factors in Definition \ref{def:CUASfam}. Moreover, this larger copy of $A_{r(\bar{e})}$ is tiled by tiles of the form $f^{-1}_{\bar{e}}\circ f_{\bar{e}'}(A_{s(\bar{e}')})$, where $\bar{e}'$ is any path from $V_0$ to $r(\bar{e})$. In other words:
\begin{equation}
  \label{eqn:approx}
  f_{\bar{e}}^{-1}(A_{r(\bar{e})}) = \bigcup_{\bar{e}'\in E_{r(\bar{e})}}f^{-1}_{\bar{e}}\circ f_{\bar{e}'}(A_{s(\bar{e}')}).
\end{equation}
Let us emphasize that the union in (\ref{eqn:approx}) is the union of copies of prototiles: for each $\bar{e}'\in E_{r(\bar{e})}$, $f^{-1}_{\bar{e}}\circ f_{\bar{e}'}(A_{s(\bar{e}')})$ is a copy of the prototile $A_{s(\bar{e}')}$. So $f_{\bar{e}}^{-1}(A_{r(\bar{e})})$ in (\ref{eqn:approx}) is a patch of some tiling since it decomposes as the union of copies of prototiles.
\begin{definition}
Let $B = B_x(\mathcal{F}_1,\dots,\mathcal{F}_N)$ be a Bratteli diagram with parameter $x$ built from a type H family $\mathcal{F}$, and pick $\bar{e}\in X_{B}$. For $k\in\mathbb{N}$, the \textbf{$k^{th}$ approximant of $\bar{e}$} is the union of tiles forming the set $f_{\bar{e}|_k}^{-1}(A_{r(\bar{e}|_k)})$ as defined in (\ref{eqn:approx}). We denote by $\mathcal{P}_k(\bar{e})$ the $k^{th}$ approximant of $\bar{e}$. By convention we make the zeroth approximant $\mathcal{P}_0(\bar{e}) := A_{s(\bar{e})}$.
\end{definition}
Since we have assumed that the shared attractor contains the origin in its interior, it follows that $\mathcal{P}_k(\bar{e})\subset \mathcal{P}_{k+1}(\bar{e})$ for any $k$. That is, the tiles in the approximant $\mathcal{P}_k(\bar{e})$ are also tiles in the approximant $\mathcal{P}_{k+1}(\bar{e})$. Thus, taking arbitrarily large values of $k$ tiles arbitrarily large parts of $\mathbb{R}^d$ through (\ref{eqn:approx}).
\begin{definition}
  Let $B = B_x(\mathcal{F}_1,\dots, \mathcal{F}_N)$ be a Bratteli diagram with parameter $x$, where $\mathcal{F}$ is a type H family, and pick $\bar{e}\in X_{B_x}$. The \textbf{tiling associated to $\bar{e}$} is
  \begin{equation}
    \label{eqn:tiling}
    \mathcal{T}_{\bar{e}} = \bigcup_{k=1}^{\infty} \mathcal{P}_k(\bar{e}).
    \end{equation}
\end{definition}
Note that the approximants $\mathcal{P}_k(\bar{e})$ are patches of $\mathcal{T}_{\bar{e}}$. Patches of the form $\mathcal{P}_k(\bar{e})$ are also called \textbf{level $k$ supertiles}.
We now investigate when it is the case the the tilings $\mathcal{T}_{\bar{e}}$ defined above cover all of $\mathbb{R}^d$ or just parts of it.

Let $\mathcal{F} = \{\mathcal{F}_1,\dots, \mathcal{F}_N\}$ be a type H family and $x\in \Sigma_N$. Let $B_x(\mathcal{F}) = (V,E)$. For each $\ell\in\mathbb{N}$, let
$$\partial V_\ell = \{\bar{e} = (e_1,\dots, e_\ell,\dots )\in X_B:f_{e_\ell}\circ \cdots \circ f_{e_1}(A_{s(e_1)})\cap \partial A_{r(e_\ell)} \neq \varnothing\},$$
and denote by
$$\limsup \partial V_k = \bigcap_{n\geq 1}\bigcup_{k\geq n}\partial V_k$$
the set of paths which are in $\partial V_k$ for infinitely many $k$.

\begin{lemma}
  \label{lem:cover}
Let $\mathcal{F}= \{\mathcal{F}_1,\dots,\mathcal{F}_N\}$ be a type H family. Then the tiling $\mathcal{T}_{\bar{e}}$ covers all of $\mathbb{R}^d$ if  $\bar{e}\not\in \limsup \partial V_k$.
\end{lemma}
\begin{proof}
  Let $\bar{e}\not\in \limsup \partial V_k$. There exists a $j$ such that $\bar{e}\not\in \partial V_k$ for all $k>j$. This means the support of $\mathcal{P}_k(\bar{e})$ is contained in the interior of the support of $\mathcal{P}_{k+1}(\bar{e})$, at a positive distance from the boundary of the support of $\mathcal{P}_{k+1}(\bar{e})$, uniformly for all $k>j$. Thus the nested approximants $\mathcal{P}_{k-1}(\bar{e})\subset \mathcal{P}_k(\bar{e})\subset \mathcal{P}_{k+1}(\bar{e})\subset \cdots$ eventually cover all of $\mathbb{R}^d$.
\end{proof}
The following result will not be used for the main theorem. However it is still interesting to know how likely it is that a path $\bar{e}\in X_B$ gives a tiling $\mathcal{T}_{\bar{e}}$ of $\mathbb{R}^d$.
\begin{lemma}
  \label{lem:limsup}
  Let $\mathcal{F}= \{\mathcal{F}_1,\dots,\mathcal{F}_N\}$ be a type H family. For $x\in\Sigma_N$ let $E$ be the set of edges of the Bratteli diagram $B = B_x(\mathcal{F})$, and assume $B_{x}$ has $M$ vertices at each level. For the two quantities
    $$\lambda^-:= \min_{v\in \{v_1,\dots, v_M\}} \left\{ \liminf_{k\rightarrow \infty}\frac{\log |E_{v}|}{k}\right\} \hspace{.4in}\mbox { and } \hspace{.4in} \lambda^+:= \max_{v\in \{v_1,\dots, v_M\}} \left\{ \limsup_{k\rightarrow \infty}\frac{\log |E_{v}|}{k}\right\}$$
  suppose that $\lambda^->0$ and
  \begin{equation}
    \label{eqn:lambdaCond}
    \lambda^+-\lambda^- < \frac{\lambda^+}{d}.
    \end{equation}
 If $\mu$ is a Borel probability measure which is invariant under the tail equivalence relation, then $\mu(\limsup \partial V_k) = 0$.
\end{lemma}
In practice, we will usually have $\lambda^+=\lambda^->0$, which will satisfy the hypotheses of the Lemma. In fact, under some mild assumptions of minimality of $B_x$, one can always show that $\lambda^+=\lambda^->0$. We leave this to the interested reader to work out.
  \begin{proof}
   Let $\mu$ be a Borel probability measure on $X_B$ which is invariant for the tail-equivalence relation. Note that
    \begin{equation}
      \label{eqn:Yk}
      \mu(\partial V_k) = \sum_{\bar{e}\in E_{0,k}} \mu(C_{\bar{e}}) = \sum_{v\in V_k}\mu(v)\cdot |E_v\cap \partial V_k|
    \end{equation}
    for any $k>0$, where $\mu(v)$ is defined in (\ref{eqn:invMeasure}). We will show that
    $$\sum_{k>0} \mu(\partial V_k)<\infty$$
  for any measure $\mu$ invariant under the tail equivalent relation. Thus by the Borel-Cantelli lemma we will have that $\mu(\limsup \partial V_k) =0$.

First, we claim that for any $\varepsilon\in( 0,\lambda^-)$ there exists a constant $c_\varepsilon>0$ such that
    \begin{equation}
      \label{eqn:firstEst}
      \mu(v) < c_\varepsilon e^{-(\lambda^--\varepsilon)k}
  \end{equation}
    for any $v\in V_k$ and $k>0$. Indeed, by the definition of $\lambda^-$, for any $\varepsilon\in(0,\lambda^-)$ there exists a $c_0$ such that for any $v\in V_k$ and $k>0$
    $$c_0e^{(\lambda^--\varepsilon)k}\leq |E_v|.$$
    Then, since $1 = \sum_{v\in V_k}\mu(v)|E_v|$, for any $k$, we have
    $$\sum_{v\in V_k }\mu(v)c_0e^{(\lambda^--\varepsilon)k}\leq \sum_{v\in V_k}\mu(v)|E_v| = 1,$$
    from which it follows that
    $$\mu(v) < \sum_{v\in V_k} \mu(v) \leq \frac{1}{c_0}e^{-(\lambda^--\varepsilon)k} = c_\varepsilon e^{-(\lambda^--\varepsilon)k} $$
    for any $k>0$ and $v\in V_k$, proving (\ref{eqn:firstEst}).

    We now claim that for any $\varepsilon\in (0,\lambda^+)$ there exists a constant $C_\varepsilon$ such that
    \begin{equation}
      \label{eqn:secondEst}
      | \partial V_k| \leq C_\varepsilon e^{(\lambda^++\varepsilon)\frac{(d-1)}{d}k}.
    \end{equation}
    Indeed, let $v\in V_k$ and let $\bar{e}\in X_B$ be such that $r(\bar{e}|_k) = v$. Then $\mathcal{P}_k(\bar{e})$ is a CW-complex of volume $\mathrm{Vol}(A_{r(\bar{e}|_k)}) (\theta_{x_k}\cdots \theta_{x_1})^{-d}$ tiled by $|E_v|$ tiles, each a copy of some prototile $t_{i}$, as in (\ref{eqn:approx}). Now, for any $\varepsilon_+\in (0,\lambda^+)$ there exists a $C_\varepsilon'>0$ such that
    $$\mathrm{Vol}(\mathcal{P}_k(\bar{e}))=\mathrm{Vol}(A_{r(\bar{e}|_k)}) (\theta_k\cdots \theta_1)^{-d} \leq C_\varepsilon' e^{(\lambda^++\varepsilon_+)k}.$$
    Let $r_*>0$ be large enough that $B_{r_*}$ contains a copy of any prototile $t_i$ in its interior.  Since $\mathcal{P}_k(\bar{e})$ is a CW-complex, then there exists a $K$ such that $\mathrm{Vol}_{}(\partial_{5r_*}( \mathcal{P}_k(\bar{e}))) \leq K \mathrm{Vol}(\mathcal{P}_k(\bar{e}))^{\frac{d-1}{d}}$. Now, since $|\partial V_k|$ is the number of paths in $E_k$ which correspond to tiles on the boundary of approximants $\mathcal{P}_k(\bar{e})$, $|\partial V_k|$ is proportional to $\mathrm{Vol}_{}(\partial_{5r_*}( \mathcal{P}_k(\bar{e})))$. Combining this with the above bound, we get (\ref{eqn:secondEst}) with $C_\varepsilon = KC_\varepsilon'$.

By (\ref{eqn:lambdaCond}) we can pick $\varepsilon_\pm \in (0,\lambda^\pm)$ small enough so that
    \begin{equation}
\lambda^+-\lambda^- + \varepsilon_+ + \varepsilon_- < \frac{\lambda^+ +\varepsilon_+}{d},
    \end{equation}
    and hence
    \begin{equation}\label{eqn:lambdaCond2}
     (d-1)\lambda^{+} - d \lambda^{-} + (d-1)\epsilon_{+} + d \epsilon^{-} < 0.
     \end{equation}

    Finally, by (\ref{eqn:Yk}), (\ref{eqn:firstEst}) and (\ref{eqn:secondEst}),
    \begin{equation*}
      \begin{split}
        \mu(\partial V_k) &= \sum_{v\in V_k}\mu(v)\cdot |E_v\cap \partial V_k| \leq |\partial V_k|\cdot M \max_{v\in V_k}\mu(v) \\
        &= M\left(C_{\varepsilon_+} e^{(\lambda^++\varepsilon_+)\frac{(d-1)}{d}k} \right)\left(c_{\varepsilon_-}e^{-(\lambda^--\varepsilon_-)k}\right)\\
        &= C' \exp\left(k\left(\frac{d-1}{d}(\lambda^+ + \epsilon_{+})-\lambda^{-}-\varepsilon_- \right)\right) \le C^{\prime}\lambda_{*}^{k}\\
      \end{split}
    \end{equation*}
    for some $\lambda_*\in(0,1)$, where we used (\ref{eqn:lambdaCond2}) in the last equality. Thus we have that $\sum_{k>0}\mu(\partial V_k)<\infty$, so by the Borel-Cantelli lemma, $\mu(\limsup \partial V_k) = 0$.
  \end{proof}
  Let $\mathcal{F}= \{\mathcal{F}_1,\dots,\mathcal{F}_N\}$ be a type H family and suppose that $B_x(\mathcal{F})$ is minimal. We define the set of singular paths by
  $$ \Sigma_{B_x} := \left\{ \bar{e}\in X_B : \mathcal{T}_{\bar{e}} \,\,\mbox{ does not tile all of } \mathbb{R}^d\right\}$$
  which, by Lemma \ref{lem:cover}, is a subset of $\limsup \partial V_k$.
\begin{definition}
The \textbf{extension set of $\mathcal{T}_{\bar{e}}$} for $\bar{e}\in \Sigma_{B_x}$ consists of all $\mathcal{T}_{\bar{e}'}$ where $\bar{e}'$ is a limit point of $\{\bar{e}_k\}$, with $\bar{e}_k\rightarrow \bar{e}\in\Sigma_{B_x}$ but $\bar{e}_k\not\in\Sigma_{B_x}$ for all $k$.
\end{definition}
That extensions exist whenever $\mu(\Sigma_{B_x})=0$ for a finite Borel invariant measure $\mu$ follows from Lemma \ref{lem:cover}: since $\bar{e}_k\not \in \Sigma_{B_x}$, $\mathcal{T}_{\bar{e}_k}$ covers all of $\mathbb{R}^d$. Since $X_B$ is compact, a limit exists along a subsequence. Thus, even if $\mathcal{T}_{\bar{e}}$ does not tile all of $\mathbb{R}^d$ for $\bar{e}\in \Sigma_{B_x}$, there are tilings of $\mathbb{R}^d$ which contain the tiling $\mathcal{T}_{\bar{e}}$, namely any of its extensions. Finally, define
$$\mathring{X}_B := X_B\backslash \Sigma_{B_x}.$$
\begin{lemma}
\label{lem:FLC}
Suppose $\mathcal{F}$ is a type H family, $x\in\Sigma_N$ and $\bar{e}\in \mathring{X}_{B_x(\mathcal{F})}$. Then $\mathcal{T}_{\bar{e}}$ has finite local complexity.
\end{lemma}
\begin{proof}
  This follows from the compatibility condition, so there are finitely many local configurations.
\end{proof}
\subsection{Topology revisited}
\label{sec:renorm}
 Let $\mathcal{F} = \{\mathcal{F}_1,\dots, \mathcal{F}_N\}$ be a type H family. Given $B_x(\mathcal{F})$ we will denote by $\theta_{x_k}$ the scaling factor of the GIFS associated with the level $k$ of the Bratteli diagram $B_x(\mathcal{F})$.
\begin{lemma}
\label{lem:twopart}
  Let $B = B_x(\mathcal{F})$ be a Bratteli diagram with parameter $x$ for a type H family, and pick $\bar{e}\in X_{B}$. Then:
\begin{enumerate}
\item if $\bar{e}_1,\bar{e}_2\in [\bar{e}]$, then there exists a $\tau$ such that $\mathcal{T}_{\bar{e}_1} = \varphi_\tau(\mathcal{T}_{\bar{e}_2})$.
\item the tiling space $\mathcal{T}_{\bar{e}}$ only depends on the minimal component in $X_{B}$ containing $\bar{e}$.
\end{enumerate}
\end{lemma}
\begin{proof}
  Let $\bar{e} \neq \bar{e}'$ in $X_{B_x}$ be tail-equivalent: there exists a smallest $k\in\mathbb{N}$ such that $\bar{e}_i = \bar{e}'_i$ for all $i>k$. Consider the approximants $\mathcal{P}_k(\bar{e})$ and $\mathcal{P}_k(\bar{e}')$. By (\ref{eqn:approx}) both approximants are the set $A_{r(\bar{e}|_k)}$ scaled by $\theta_{x_1}^{-1}\cdots \theta_{x_k}^{-1}$ and are tiled by tiles in bijection with paths from $V_0$ to $r(\bar{e}|_k)=r(\bar{e}'|_k)$ in the same way. Thus there is a $\tau\in\mathbb{R}^d$ such that $\mathcal{P}_k(\bar{e}) = \varphi_\tau(\mathcal{P}_k(\bar{e}'))$ and $\partial\mathcal{P}_k(\bar{e}) = \varphi_\tau(\partial\mathcal{P}_k(\bar{e}'))$.

  The fact that $\bar{e}_i = \bar{e}'_i$ for all $i>k$ means that heirarchical structures $\mathcal{P}_{k}(\bar{e})\subset \mathcal{P}_{k+1}(\bar{e})\subset \mathcal{P}_{k+2}(\bar{e})\subset\cdots$ and $\mathcal{P}_{k}(\bar{e}')\subset \mathcal{P}_{k+1}(\bar{e}')\subset \mathcal{P}_{k+2}(\bar{e}')\subset\cdots$ are the same. Thus, in the limit, $\mathcal{T}_{\bar{e}}$ differs from $\mathcal{T}_{\bar{e}'}$ by a translation: $\mathcal{T}_{\bar{e}} = \varphi_\tau(\mathcal{T}_{\bar{e}'})$, which proves the first part.

  By (i), $[\bar{e}]$ can be identified with a set of translates $\varphi_{\tau_{\bar{e}'}}(\mathcal{T}_{\bar{e}})$ of $\mathcal{T}_{\bar{e}}$, for some $\tau_{\bar{e}'}$ depending on $\bar{e}'\in[\bar{e}]$. So $\Omega_{\mathcal{T}_{\bar{e}'}} = \Omega_{\mathcal{T}_{\bar{e}}} $ whenever $[\bar{e}] = [\bar{e}']$. Let $\bar{e}'\in\overline{[\bar{e}]}$ but $\bar{e}'\not\in [\bar{e}]$. Then there exists a sequence $\{\bar{e}^k\}$ in $[\bar{e}]$ converging to $\bar{e}'$ in $\overline{[\bar{e}]}$ with $\bar{e}^k_i = \bar{e}'_i$ for all $i\leq k$. This means that $\mathcal{P}_k(\bar{e}^k) = \mathcal{P}_k(\bar{e}')$ as tiled patches for all $k\in\mathbb{N}$. Thus $d(\mathcal{T}_{\bar{e}'},\mathcal{T}_{\bar{e}^k}) = d(\mathcal{T}_{\bar{e}'}, \varphi_{\tau_{\bar{e}^k}}(\mathcal{T}_{\bar{e}}))\rightarrow 0$, so $\mathcal{T}_{\bar{e}'}\in \Omega_{\mathcal{T}_{\bar{e}}}$.
\end{proof}
Given $\bar{e}\in \mathring{X}_{B_x(\mathcal{F})}$, recall the definition of the tiling $\mathcal{T}_{\bar{e}}$ as defined in (\ref{eqn:tiling}). Since we are assuming that the shared attractor of $\mathcal{F}$ contains the origin, then the origin is contained in the interior of every single prototile, which we can treat as a distinguished point. As such, we have that if $\mathcal{T}_{\bar{e}}$ tiles all of $\mathbb{R}^d$ then $\mathcal{T}_{\bar{e}}\in \mho_{\mathcal{T}_{\bar{e}}}$. Let $\Delta_x:\mathring{X}_{B_x(\mathcal{F})}\rightarrow \mho_{\mathcal{T}_{\bar{e}}}$ be the map $\Delta_x(\bar{e}) = \mathcal{T}_{\bar{e}}$ . This is called the Robinson map in \cite{kellendonk:gap} where the following type of result can be found.
\begin{proposition}
  \label{prop:Robinson1}
  Let $\mathcal{F}= \{\mathcal{F}_1,\dots,\mathcal{F}_N\}$ be a type H family and suppose that $B_x(\mathcal{F})$ is minimal. The Robinson map $\Delta_x:\mathring{X}_{B_x}\rightarrow \mho_x$ is a continuous map onto its image which defines a bijection between Borel probability measures $\mu$ on $X_{B_x(\mathcal{F})}$ which are invariant for the tail-equivalence relation and satisfy $\mu(\Sigma_{B_x})=0$ with Borel transverse invariant probability measures for the $\mathbb{R}^d$ action on $\Omega_{\mathcal{T}_{\bar{e}}}$ supported on the canonical transversal $\mho_{\mathcal{T}_{\bar{e}}}$.
\end{proposition}
\begin{proof}
  Since $B_x$ is minimal, by the second part of Lemma \ref{lem:twopart}, the tiling space is independent of which $\bar{e}\in\mathring{X}_{B_x}$ is used to construct a tiling $\mathcal{T}_{\bar{e}}$ and then a tiling space. As such, the canonical transversal $\mho_x$ is also independent of this choice. That the map $\Delta_x$ is a continuous surjection onto $\mho_x$ follows directly by considering sequences in $\mathring{X}_{B_x}$, their approximants, and their images through $\Delta_x$.


  Now consider a Borel probability measure $\mu$ which is invariant under the tail-equivalence relation with $\mu(\Sigma_{B_x}) = 0$,.
By definition, we have that if $\bar{e}$ and $\bar{e}'$ have the property that $s(\bar{e}),s(\bar{e}')\in V_0$ and $r(\bar{e})=r(\bar{e}')$, then $\mu(C_{\bar{e}})=\mu(C_{\bar{e}'})$. A Borel probability measure $\nu$ on the canonical transversal is invariant under the $\mathbb{R}^d$ action if for any open set $C\subset \mho_x$ and vector $\tau$ with $\varphi_\tau(C)\subset \mho_x$ we have that $\nu(\varphi_\tau(C)) = \nu(C)$. Thus we verify the pushforward of $(\Delta_x)_*\mu$ on open sets of the form $C(\mathcal{P}_k(\bar{e}))$ and $C(\mathcal{P}_k(\bar{e}'))$ with $s(\bar{e}),s(\bar{e}')\in V_0$ and $r(\bar{e}) = r(\bar{e}')$:
  \begin{equation*}
    \begin{split}
      (\Delta_x)_*\mu(C(\mathcal{P}_k(\bar{e}))) &= \mu(\Delta_x^{-1}(C(\mathcal{P}_k(\bar{e})))) = \mu(C_{\bar{e}})  =  \mu(C_{\bar{e}'}) =\mu(\Delta_x^{-1}(C(\mathcal{P}_k(\bar{e}'))))\\
      &= (\Delta_x)_*\mu(C(\mathcal{P}_k(\bar{e}'))) = (\Delta_x)_*\mu(\varphi_\tau(C(\mathcal{P}_k(\bar{e})))),
      \end{split}
  \end{equation*}
  where we used the invariance of $\mu$ in the third equality and part (i) of Lemma \ref{lem:twopart} in the last one. So $(\Delta_x)_*\mu$ is invariant for the $\mathbb{R}^d$ action. That the inverse $\Delta_x^{-1}$ sends invariant measures to invariant measures for which $\Sigma_{B_x}$ is a null set is similarly proved.
\end{proof}

\section{Bi-infinite diagrams and hierarchical structures}
\label{sec:biinf}
We now extend the construction of tilings using Bratteli diagrams to bi-infinite Bratteli diagrams. A key application of this construction appears in Proposition \ref{prop:conj}, where, using Proposition \ref{prop:Robinson1}, we connect the shift map $\sigma$ on the parameter space $\Sigma_{N}$ (the full shift on $N$ symbols) to an induced map between tiling spaces $\Omega_{x} \to \Omega_{\sigma(x)}$. We follow the conventions of \cite{LT,trevino:masur}. Recall that $\bar{\mathbb{Z}} = \mathbb{Z}-\{0\}$, and note $\bar{\mathbb{Z}}$ inherits an order from the order on $\mathbb{Z}$.
\subsection{Bi-infinite diagrams}
A \textbf{bi-infinite Bratteli diagram} $\mathcal{B}=(\mathcal{V},\mathcal{E})$ is an infinite graph with vertex and edge sets partitioned as
$$\mathcal{V} = \bigsqcup_{k\in\mathbb{Z}}\mathcal{V}_k\hspace{.3in}\mbox{ and }\hspace{.3in}\mathcal{E} = \bigsqcup_{k\in\bar{\mathbb{Z}}}\mathcal{E}_k,$$
along with range and source maps $r,s:\mathcal{E}\rightarrow\mathcal{V}$ which are defined as $r:\mathcal{E}_k\rightarrow \mathcal{V}_{k}$ for $k>0$ and $s:\mathcal{E}_k\rightarrow \mathcal{V}_{k}$ for $k<0$, while $s:\mathcal{E}_k\rightarrow \mathcal{V}_{k-1}$ for $k>0$ and $r:\mathcal{E}_k\rightarrow \mathcal{V}_{k+1}$ for $k<0$. The definitions of paths from \S \ref{subsec:brat} are generalized in the natural way to the bi-infinite case.

For a bi-infinite Bratteli diagram $\mathcal{B}$, we denote by $X_{\mathcal{B}}$ the set of infinite paths in $\mathcal{B}$, i.e.,
$$X_{\mathcal{B}} = \left\{\bar{e} = (\dots, e_{k-1},e_k,e_{k+1},\dots) \in \prod_{k\in\bar{\mathbb{Z}}}\mathcal{E}_k: r(e_i) = s(e_{i+1})\mbox{ for all }i\in\bar{\mathbb{Z}}\right\}.$$
The topology of $X_{\mathcal{B}(\mathcal{F})}$ is generated by cylinder sets of the form $C_{\bar{e}}$, where $\bar{e}$ is a finite path in $\mathcal{B}(\mathcal{F})$.
\begin{definition}
Let $\mathcal{B}$ be a bi-infinite Bratteli diagram. The \textbf{positive part of $\mathcal{B}$}, denoted by $\mathcal{B}^+$, is the (not bi-infinite) Bratteli diagram $B$ where the vertices, edges, and source and range maps are the same as those of $\mathcal{B}$ when we restrict to sets with non-negative indices. Likewise, the \textbf{negative part of $\mathcal{B}$}, denoted by $\mathcal{B}^-$, is obtained by restricting to sets with negative indices ignoring all the sets with positive indices in $\mathcal{B}$ and then reversing the sign of the indices of the sets left.
\end{definition}

\subsection{Hierarchical structures}
Let $\Sigma_N = \{1,\dots, N\}^{\bar{\mathbb{Z}}}$, where $\bar{\mathbb{Z}} := \mathbb{Z}-\{0\}$, inheriting an order from that of $\mathbb{Z}$. Given a type H family $\mathcal{F} = \{\mathcal{F}_1,\dots, \mathcal{F}_N\}$ and $x\in \Sigma_N$ we consider the bi-infinite Bratteli diagram $\mathcal{B}_{x}(\mathcal{F})$ by defining its $k^{th}$ transition matrix $\mathcal{M}_k$ to be $\mathcal{M}(\mathcal{F}_i)$ if $x_k = i$.

Recall that, assuming $\mathcal{B}^+_x$ is minimal, paths starting at $V_0$ in $\mathcal{B}^+_x$ record the hierarchical structure of a tiling in $\mho_x$. This is done through the approximants: if $\bar{e}\in \mathring{X}_{\mathcal{B}^+_x}$, then the hierarchical structure of the tiling $\Delta_x(\bar{e})$ is described by the inclusions
$$\{0\} \subset A_{s(\bar{e})}\subset \mathcal{P}_1(\bar{e})\subset \cdots \subset \mathcal{P}_k(\bar{e})\subset \mathcal{P}_{k+1}(\bar{e})\subset \cdots.$$

Following the philosophy of \cite{BowenMarcus}, paths in the negative part of $\mathcal{B}_x$ describe the transverse structure of the object described by paths in the positive part. By Proposition \ref{prop:Robinson1}, paths in $\mathcal{B}^+_x$ describe the local structure of $\mho_x$, so considering the local product structure of $\Omega_x$ in Proposition \ref{prop:productStruct}, then the paths in $\mathcal{B}^-_x$ ought to describe the local structure of the leaves which foliate $\Omega_x$. We now describe how this is done.

Let $\bar{e}$ be a path with $s(\bar{e})\in \mathcal{V}_k$, $r(\bar{e})\in \mathcal{V}_0$ and $k<0$. As in (\ref{eqn:comp}), there is a map $f_{\bar{e}}$ which maps $A_{s(\bar{e})}$ into $A_{r(\bar{e})}$. In fact, as in (\ref{eqn:approx}), for any $k<0$ and $v\in\mathcal{V}_0$, the prototile $A_{v}$ is tiled by tiles indexed by all paths $\bar{e}$ with $s(\bar{e})\in \mathcal{V}_k$ and $r(\bar{e}) = v$:
$$A_v = \bigcup_{\{\bar{e}:s(\bar{e})\in\mathcal{V}_k, r(\bar{e}) = v \}} f_{\bar{e}}(A_{s(\bar{e})}).$$
Since all maps $f$ in the family $\mathcal{F}$ are affine and contracting, the prototile $A_v$ can be partitioned by smaller and smaller tiles by considering longer and longer paths ending in $v\in \mathcal{V}_0$. As such, any point in $A_v$ has a (not necessarily unique) address given by an infinite path in the negative part of $\mathcal{B}$, given by a surjective function $p_x:X_{\mathcal{B}^-_x}\rightarrow \bigcup_{v\in\mathcal{V}_0}A_v$. Define
$$\mathring{X}_{\mathcal{B}_x}:= \{\bar{e} = (\dots, e_{-2},e_{-1},e_1,e_2,\dots)\in X_{\mathcal{B}_x}: \bar{e}^+:=(e_1,e_2,\dots) \not \in \Sigma_{\mathcal{B}^+_x}\}.$$
We get the extension of Proposition \ref{prop:Robinson1} in the bi-infinite case.
\begin{proposition}
  \label{prop:Robinson2}
  The Robinson map $\Delta_x:\mathring{X}_{\mathcal{B}^+_x}\rightarrow \mho_x$ extends to a continuous map $\bar{\Delta}_x:\mathring{X}_{\mathcal{B}_x}\rightarrow \Omega_x$.
\end{proposition}
\begin{proof}
Let $\bar{e} = (\dots, e_{-2}, e_{-1}, e_1, e_2, \dots)\in \mathring{X}_{\mathcal{B}_x}$ and denote by $\bar{e}^+\in \mathring{X}_{\mathcal{B}^+_x}$ its restriction to $\mathcal{B}^+_x$. Proposition \ref{prop:Robinson2} assigns to every path $\bar{e}^+\not\in \Sigma_{\mathcal{B}_x^+}$ in the positive part of $\mathcal{B}$ a unique tiling $\Delta_x(\bar{e}^+) = \mathcal{T}_{\bar{e}^+}$ in $\mho_x$ where the origin is contained in the interior of the tile containing the origin (which is $A_{s(\bar{e}^+)}$). So considering the positive part of $\bar{e}$ we know which tiling in the canonical transversal we obtain. As described in the paragraph above, an infinite path $\bar{e}^-$ in the negative part (that is, terminating in $\mathcal{V}_0$) defines a point $p_x(\bar{e}^-)\in A_{r(\bar{e}^-)}$. Since the prototile corresponding to the range of the negative part of the path $\bar{e}$ is the prototile containing the origin given by the positive part, we can translate the tiling $\mathcal{T}_{\bar{e}^+}$ by a small vector so that the origin can be identified with the point $p_x(\bar{e}^-)\in A_{r(e_{-1})} = A_{s(e_{1})}$. This assignment can be seen to be continuous.
\end{proof}
\subsection{Renormalization}
\label{subsec:renorm}
Recall $\Sigma_N = \{1,\dots, N\}^{\bar{\mathbb{Z}}}$, where $\bar{\mathbb{Z}} := \mathbb{Z}-\{0\}$, inheriting an order from that of $\mathbb{Z}$. In that case, $\sigma:\Sigma_N\rightarrow \Sigma_N$ denotes the full $N$-shift, obtained by shifting the labels by one in the entries of $x\in\Sigma_N$. Let $\bar{\sigma}:X_{\mathcal{B}_x}\rightarrow X_{\mathcal{B}_{\sigma(x)}}$ denote the continuous map sending a path $\bar{e}\in X_{\mathcal{B}_x}$ to itself in $X_{\mathcal{B}_{\sigma(x)}}$, where it is viewed with different indices.
\begin{definition}
  Let $\mathcal{F} = \{\mathcal{F}_1,\dots, \mathcal{F}_N\}$ be a type H family. A $\sigma$-invariant ergodic probability measure $\mu$ on $\Sigma_N$ is \textbf{minimal with respect to $\mathcal{F}$} if $B_x(\mathcal{F})$ is minimal for $\mu$-almost every $x$.
\end{definition}
If $\mu$ is minimal with respect to some type H family $\mathcal{F}$ we will only say that it is minimal when it is clear from context that we refer to $\mathcal{F}$.
\begin{remark}
  Note that if each substitution rule $\mathcal{F}_1,\dots,\mathcal{F}_N$ is primitive, then we should expect $B_x(\mathcal{F})$ to be minimal. In such case any ergodic $\sigma$-invariant probability measure will be minimal with respect to $\mathcal{F}$. The definition becomes interesting when not all $\mathcal{F}_1,\dots, \mathcal{F}_N$ are primitive substitutions, but enough random combinations of them give minimal Bratteli diagrams.
\end{remark}
Note that being minimal is a $\sigma$-invariant condition: if $B_x(\mathcal{F})$ is minimal then so is $B_{\sigma(x)}(\mathcal{F})$. As such, the set of minimal $B_x(\mathcal{F})$ is $\sigma$-invariant, so they have either full or zero measure for any ergodic invariant probability measure $\mu$. This observation, combined with the Poincar\'e recurrence theorem and the main result of \cite{trevino:masur}, gives the following.
\begin{proposition}
  \label{prop:UE}
Let $\mathcal{F} = \{\mathcal{F}_1,\dots, \mathcal{F}_N\}$ be a type H family and let $\mu$ be an minimal, ergodic $\sigma$-invariant probability measure on $\Sigma_N$. Then for $\mu$-almost every $x\in\Sigma_N$ the $\mathbb{R}^d$ action on $\Omega_x$ is uniquely ergodic.
\end{proposition}

Let $\Phi_{\bar{e}}$ be the map which assigns to the tiling $\mathcal{T}_{\bar{e}}$ the tiling $\mathcal{T}_{\sigma(\bar{e})}$. This extends to a nice map on the tiling space of $\mathcal{T}_{\bar{e}}$.
\begin{proposition}
\label{prop:conj}
Let $\mathcal{F} = \{\mathcal{F}_1,\dots, \mathcal{F}_N\}$ be a type H family and let $x\in\Sigma_N$ be such that $\mathcal{B}_x^+(\mathcal{F})$ is minimal. The shift map $\sigma:\Sigma_N\rightarrow\Sigma_N$ induces a homeomorphism $\Phi_{x}: \Omega_{x} \rightarrow \Omega_{\sigma(x)}$ which satisfies the conjugacy equation
\begin{equation}
  \label{eqn:conj}
  \Phi_{x}\circ\varphi_t = \varphi_{\theta_{x_1} t}\circ\Phi_{x}.
\end{equation}
The hierarchical structure shifts under the map $\Phi_x$: if $t^{(1)}$ is a level-1 supertile in some tiling  $\mathcal{T}\in\Omega_x$ then $\theta_{x_1}t$ is a tile in the tiling $\Phi_x(\mathcal{T})$.
\end{proposition}
\begin{proof}
  Let us first describe the inverse $\Phi_x^{-1}$. Take a tiling $\mathcal{T}\in \Omega_{\sigma(x)}$. The image $\Phi_x^{-1}(\mathcal{T})$ is the tiling $\mathcal{T}'\in\Omega_x$ obtained from $\mathcal{T}$ by substituting each tile in $\mathcal{T}$ using the substitution rule $\mathcal{F}_{(\sigma(x))_{-1}} = \mathcal{F}_{x_{1}}$ and rescaling by $\theta^{-1}_{(\sigma(x))_{-1}} = \theta^{-1}_{x_1}$. Thus the map $\Phi_x^{-1}$ adds the smallest level of hierarchical structure and can easily be seen to be continuous.

  As such, the map $\Phi_x$ should remove the smallest level of the hierarchical structure. So if $\mathcal{T}\in\Omega_x$ then $\Phi_x(\mathcal{T})$ is the tiling in $\Omega_x$ obtained by first erasing all level-0 supertiles in $\mathcal{T}$, leaving a tiling of $\mathbb{R}^d$ where the tiles are the level-1 supertiles of $\mathcal{T}$. Rescaling this tiling by $\theta_{x_1}$ gives us the tiling $\Phi_x(\mathcal{T})$. This is also easily seen to be continuous, and the shifting of hierarchical structures follows from this.

  From these operations one can see that the conjugacy equation (\ref{eqn:conj}) holds; we leave the details to the reader.
\end{proof}

\section{Cohomology}
\label{sec:AP}
Let $\mathcal{T}$ be a tiling of $\mathbb{R}^d$. For $R > 0$, a function $f:\mathbb{R}^d\rightarrow \mathbb{R}$ is called \textbf{$\mathcal{T}$-equivariant of range $R$} if
$$\varphi_x(\mathcal{T})\cap B_R(0) = \varphi_y(\mathcal{T})\cap B_R(0) \Rightarrow f(x) = f(y).$$
We say $f$ is \textbf{$\mathcal{T}$-equivariant} if it is $\mathcal{T}$-equivariant of range $R$ for some $R > 0$. The set of all $\mathcal{T}$-equivariant, $C^\infty$ functions is denoted by $\Delta_{\mathcal{T}}^0$. A $k$-form $\eta$ is $\mathcal{T}$-equivariant if each function involved in $\eta$ is $\mathcal{T}$-equivariant, and we denote the set of all $\mathcal{T}$-equivariant, smooth $k$-forms by $\Delta_{\mathcal{T}}^k$.

The complex $\{\Delta_{\mathcal{T}}^k,d\}$ is a subcomplex of the de Rham complex of smooth differential forms. As such, the restriction of $d$ to $\mathcal{T}$-equivariant forms satisfies $d^2 = 0$. We can define its cohomology by
\begin{equation}
  \label{eqn:PEC}
  H^k(\Omega_{\mathcal{T}};\mathbb{R}) = \frac{\mathrm{Ker}\{d:\Delta_{\mathcal{T}}^k\rightarrow\Delta^{k+1}_\mathcal{T}\}}{\mathrm{Im}\{d:\Delta_{\mathcal{T}}^{k-1}\rightarrow\Delta^{k}_\mathcal{T}\}},
  \end{equation}
which is the \textbf{$\mathcal{T}$-equivariant cohomology} of $\Omega_\mathcal{T}$.

\subsection{The Anderson-Putnam Complex}
\label{subsec:AP}
Let $\Omega$ be a tiling space. For any tile $t$ in the tiling $\mathcal{T}\in\Omega$, the set $\mathcal{T}(t)$ denotes all tiles in $\mathcal{T}$ which intersect $t$. This type of patch is called a \textbf{collared tile}.
\begin{definition}
  \label{def:AP}
  Let $\Omega$ be a tiling space. Consider the space $\Omega\times \mathbb{R}^d$ under the product topology, where $\Omega$ carries the discrete topology and $\mathbb{R}^d$ the usual topology. Let $\sim_1$ be the equivalence relation on $\Omega\times\mathbb{R}^d$ which declares a pair $(\mathcal{T}_1,u_1)\sim_1 (\mathcal{T}_2,u_2)$ if $\mathcal{T}_1(t_1)-u_1 = \mathcal{T}_2(t_2)-u_2$ for some tiles $t_1,t_2$ with $u_1\in t_1\in\mathcal{T}_1$ and $u_2\in t_2\in\mathcal{T}_2$. The space $(\Omega\times\mathbb{R}^d)/\sim_1$ is called the \textbf{Anderson-Putnam (AP) complex} of $\Omega$ and is denoted by $AP(\Omega)$.
\end{definition}
Let us now review the AP-complexes involved in our construction. First, note that for $x  = (x_1,x_2,\dots)\in \Sigma_N$, assuming that $B_x(\mathcal{F})$ is minimal, $AP(\Omega_x(\mathcal{F}))$ only depends on finitely many symbols $x_1,\dots, x_\ell$, since the collaring of tiles in tilings of $\Omega_x$ only depends in the $k^{th}$-approximants $\mathcal{P}_k$ for sufficiently large $k$. Thus for a type H family $\mathcal{F}$, there exists a partition $U_1,\dots, U_q$ of $\Sigma_N$ by open sets and CW-complexes $\Gamma_1,\dots, \Gamma_q$ such that if $x\in U_i$ and $B_x$ is minimal, then $AP(\Omega_x(\mathcal{F})) = \Gamma_i$.

It will be useful to also consider higher level AP complexes, defined as follows. For $x\in \Sigma_N$, $\mathcal{T}\in\Omega_x$, and denoting by $t^{(1)}$ a level 1 supertile (which are approximants of the form $\mathcal{P}_1(\bar{e})$ as in (\ref{eqn:approx})), let $\mathcal{T}(t^{(1)})$ be the union of the supertile $t^{(1)}$ and the level 1 supertiles of $\mathcal{T}$ which intersect $t^{(1)}$. Proceeding similarly, we let $\mathcal{T}(t^{(n)})$ denote the \textbf{collared level-$n$ supertile} corresponding to the level-$n$ tupertile $t^{(n)}$. Let $\sim_n$ be the equivalence relation defined by the collared level-$n$ supertiles $\mathcal{T}(t^{(n)})$ in $\Omega$ as in Definition \ref{def:AP}. The quotient $(\Omega\times\mathbb{R}^d)/\sim_n$ is denoted by $AP^n(\Omega)$. By construction, $AP(\Omega_{\sigma^n(x)})$ and $AP^n(\Omega_x)$ are homeomorphic. In fact, the only difference is their scale: $AP^n(\Omega_x)$ is a rescaling of $AP(\Omega_{\sigma^n(x)})$ by $\theta_{x_n}^{-1}\cdots \theta_{x_1}^{-1}$. Denote by
\begin{equation}
  \label{eqn:rescale}
  r_{k,x}:AP^k(\Omega_x)\rightarrow AP(\Omega_{\sigma^k(x)})
\end{equation}
the rescaling homeomorphisms.
\begin{proposition}
  The substitution rule $\mathcal{F}_{x_i}$ induces a continuous map $\gamma_i:AP(\Omega_{\sigma^i(x)}(\mathcal{F}))\rightarrow AP(\Omega_{\sigma^{i-1}(x)}(\mathcal{F}))$ defined by $\gamma_i(\mathcal{T},u) = (\mathcal{F}_{x_i}(\mathcal{T}),\theta_{x_i}^{-1}u)$ for all $i>0$.
\end{proposition}
\begin{proof}
  \cite[Proposition 4.2]{AP}
\end{proof}
We now want to relate the tiling spaces to inverse limit constructions as first done by Anderson and Putnam in \cite{AP}.
\begin{theorem}
  \label{thm:inverseLimit}
Let $x\in\Sigma_N$ be such that $X_{B_x(\mathcal{F})}$ is minimal. Then
\begin{equation}
  \label{eqn:invLim}
  \Omega_{x} = \varprojlim_{\gamma_i} AP(\Omega_{\sigma^i(x)}(\mathcal{F})).
  \end{equation}
\end{theorem}
\begin{proof}
  \cite[Proposition 4.3]{AP}
\end{proof}
\begin{theorem}
  \label{thm:APcohomology}
  Let $x\in\Sigma_N$ be such that $X_{B_x(\mathcal{F})}$ is minimal. The \v Cech cohomology groups of $\Omega_{x}(\mathcal{F})$ are the direct limits
  \begin{equation}
    \label{eqn:directLimit}
    \check{H}^i(\Omega_x(\mathcal{F});\mathbb{R}) = \varinjlim_{\gamma_k^*}\check{H}^i(AP(\Omega_{\sigma^k(x)}(\mathcal{F}));\mathbb{R}).
    \end{equation}
\end{theorem}
\begin{proof}
\cite[Theorem 6.1]{AP}
\end{proof}

\begin{proposition}
  \label{prop:ES}
  Let $V_1,\dots, V_n$ be finite dimensional vector spaces and $\gamma^*_{i,j}:V_j\rightarrow V_i$ be linear maps. Then
    $$W_x^+:= \varinjlim_{k\geq 0} (V_{x_k},\gamma^*_{x_k,x_{k+1}})$$
is finite dimensional. Moreover, for any $\sigma$-invariant ergodic probability measure $\mu$ on $(\Sigma_{n},\sigma)$ we have that for $\mu$-almost every $x\in \Sigma_n$, there exists a subspace $ES_x \subset V_{x_{0}}$ such that the map $(\gamma_{x_{0}}^{*})^{\infty} \colon V_{x_{0}} \to W_{x}^{+}$ (given by the definition of the direct limit) takes $ES_{x}$ onto $W_{x}^{+}$.
\end{proposition}

\begin{proof}
  First we show $W_{x}^{+}$ is finite dimensional. Let $N = \max \mathrm{dim}\,V_{i}$. We claim that $\dim W_x^+ \le N$. To see this, suppose we have $k > N$ vectors $[v_{i},k(i)] \in W_x^+$, and let $K = \underset{i}\max\, k(i)$. Then there exists vectors $v_{i}^{\prime} \in V_{K}$ such that $[v_{i},k(i)] = [v_{i}^{\prime},K]$ in $W_x^+$. Since dim$V_{K} \le N$, the set $\{v_{i}^{\prime}\}$ is linearly dependent, and hence so are the $[v_{i},k(i)]$ in $W_x^+$.
  
 To prove the second part, let $W$ be the set of finite words from an alphabet of $n$ symbols.  For any word $w = w_0w_1 \in W$ of length 2, let $\gamma^*_w= \gamma_{w_1,w_0}^*:V_{w_0}\rightarrow V_{w_1}$. Now for any word $w = w_0w_1\cdots w_k\in W$ of length $k+1\geq 3$, define $\gamma^*_w = \gamma_{w_kw_{k-1}}^*\cdots \gamma_{w_2w_{1}}^*\gamma_{w_1w_{0}}^*:V_{w_0}\rightarrow V_{w_k}$. For every finite word $w$ denote the cylinder set
  $$C_w = \{x\in\Sigma_n: x_0\cdots x_{|w|-1} = w \}$$
and define
$$W_\mu := \{w\in W: \mu(C_{w})>0\},\hspace{.25in} R_\mu := \min_{w\in W_\mu}\mbox{rank}\, \gamma^*_w \hspace{.25in}\mbox{ and } \hspace{.25in}L_\mu := \min_{\substack{w\in W_\mu: \\ \mathrm{rank}\, \gamma_w^* = R_\mu}} |w|.$$
It follows from Poincar\'e recurrence that if
$$A_\mu' = \{x\in \Sigma_n: \mbox{ any word $w$ found in $x$ is an element of $W_\mu$}\}$$
  and $A_\mu = A'_\mu \cap \mathrm{supp}\, \mu$ then $\mu(A_\mu) = 1$. For $x = (\dots, x_{-1},x_{0},x_1,x_2,\dots)\in \Sigma_n$ we denote the associated one-sided infinite string by $x^+ = (x_0,x_1,x_2,\dots)$. Then for all $x\in A_\mu$ we can write it (non-uniquely) as the concatenation of finite words $x^+ = a_1^xb_1^xa_2^xb_2^xa_3^xb_3^x\cdots$ where $a_i^x$ is some finite, possibly empty word, and $b_i^x$ satisfies $\mathrm{rank}\, \gamma_{b_i^x}^* = R_\mu$ and $|b_i^x| = L_\mu$ for all $i\geq 0$.

  Now take $x\in A_\mu$ and let $c\in W^+_x$. Recall that this is an equivalence class: two elements $c_i\in V_{x_i}$ and $c_j\in V_{x_j}$ are equivalent if there exists a $k\geq \max\{i,j\}$ such that $\gamma_{x_i\cdots x_k}^*c_i = \gamma_{x_j\cdots x_k}^*c_j\in V_{x_k} $. So pick a representative $c_\ell\in V_{x_\ell}$ of $c\in W_x^+$. Then fixing a decomposition $x^+ = a_1^xb_1^xa_2^xb_2^xa_3^xb_3^x\cdots$ we have that $x_\ell$ is in either a $a_i^x$ or $b_i^x$ for some $i$. Suppose it is in some $a_i^x$ (the other case is similarly treated). Then there is a $\ell'>\ell$ with $\ell'-\ell\leq L_\mu + |a_i^x|$ such that $\mathrm{rank}\, \gamma_{\ell',\ell}^* = R_\mu$. Similarly, we have that $\mathrm{rank}\, \gamma_{\ell',0}^* = R_\mu$, since $R_{\mu}$ was defined to be the minimal such rank which appears. Thus $\gamma_{\ell',\ell}^*c_\ell \in \mathrm{Im}\, (\gamma_{\ell',\ell}^*) = \mathrm{Im}\, (\gamma_{\ell',0}^*)$, so there exists a $c_0\in V_{x_0}$ such that $\gamma_{\ell',\ell}^*c_\ell = \gamma_{\ell',0}^*c_0$, and so $c$ has a representative in $V_{x_0}$. Since any $c\in W_x^+$ has a representative in $V_{x_0}$ and this forms a vector space, the space of representatives in $V_{x_0}$ of $W_x^+$ is denoted by $ES_x$.
\end{proof}

\begin{definition}
  A continuous function $f:\Omega_\mathcal{T}\rightarrow \mathbb{R}$ is called \textbf{transversally locally constant} if for any $\mathcal{T}'\in\Omega_\mathcal{T}$ there exists a $R>0$ such that
  $$B_R(0)\cap \mathcal{T}' = B_R(0)\cap \mathcal{T}''\, \Longrightarrow f(\mathcal{T}') = f(\mathcal{T}'').$$
\end{definition}
We denote by $C^\infty_{tlc}(\Omega_\mathcal{T})$ the set of all transversally locally functions on $\Omega_\mathcal{T}$ which are $C^\infty$ along the leaves of the foliation of $\Omega_\mathcal{T}$. The map $i_\mathcal{T}:C^{\infty}_{tlc}(\Omega_\mathcal{T})\longrightarrow \Delta^0_{\mathcal{T}}$ defined by
\begin{equation}
  \label{eqn:algIso}
i_\mathcal{T}(f)(t) = f\circ \varphi_t(\mathcal{T}).
\end{equation}
extends to an isomorphism $i_\mathcal{T}:C^\infty_{tlc}(\Omega_\mathcal{T},\wedge^k\mathbb{R}^d)\rightarrow\Delta_\mathcal{T}^k$. As such, we can define the Hodge map $\star:C^\infty_{tlc}(\Omega_\mathcal{T},\wedge^k\mathbb{R}^d)\rightarrow C^\infty_{tlc}(\Omega_\mathcal{T},\wedge^{d-k}\mathbb{R}^d)$ through $\star\eta = i_\mathcal{T}^{-1}\star i_\mathcal{T}\eta$.


We can define the leaf-wise derivative using the $\mathbb{R}^{d}$-action on $\Omega_{\mathcal{T}}$ as follows: for any vector $v\in\mathbb{R}^d$ we have
$$v:f\mapsto vf := \lim_{t\rightarrow 0} \frac{f \circ \varphi_{vt} - f}{t},$$
for any $f\in C^\infty_{tlc}(\Omega_\mathcal{T})$. This extends to maps $v:C^\infty_{tlc}(\Omega_\mathcal{T},\wedge^k\mathbb{R}^d)\rightarrow C^\infty_{tlc}(\Omega_\mathcal{T},\wedge^k\mathbb{R}^d)$ and in particular yields the differential $d:C^\infty_{tlc}(\Omega_\mathcal{T},\wedge^k\mathbb{R}^d)\rightarrow C^\infty_{tlc}(\Omega_\mathcal{T},\wedge^{k+1}\mathbb{R}^d)$ satisfying $d^2=0$. The cohomology of the complex $\{C^\infty_{tlc}(\Omega_\mathcal{T},\wedge^k\mathbb{R}^d)\}$ is the \textbf{foliated cohomology} of $\Omega_\mathcal{T}$ and it is denoted by $H^*(\mathbb{R}^d,C^{\infty}_{tlc}(\Omega_\mathcal{T},\mathbb{R}))$. The map (\ref{eqn:algIso}) is an isomorphism of the complexes $(\Delta_\mathcal{T}^k,d)$ and $(C^\infty_{tlc}(\Omega_\mathcal{T},\wedge^k\mathbb{R}^d),d)$ which intertwines the actions of the differentials. We summarize the above in a theorem, which can also be found in \cite[Theorem 23]{Kellendonk-Putnam:RS}.
\begin{theorem}
  \label{thm:algIso}
  The map $i_\mathcal{T}$ (\ref{eqn:algIso}) is an algebra isomorphism between the transversally locally functions which are smooth along leaves and smooth $\mathcal{T}$-equivariant functions. It yields an isomorphism of the foliated cohomology $H^*(\mathbb{R}^d,C^{\infty}_{tlc}(\Omega_\mathcal{T},\mathbb{R}))$ and the $\mathcal{T}$-equivariant cohomology $H^*(\Omega_\mathcal{T},\mathbb{R})$.
\end{theorem}
There is also a relationship between the \v{C}ech cohomology $\check{H}^*(\Omega_\mathcal{T};\mathbb{R})$ and the $\mathcal{T}$-equivariant cohomology. The following is found in \cite[Theorem 20]{Kellendonk-Putnam:RS}, or \cite{Sadun:PE}.
\begin{theorem}
  \label{thm:PatternCech}
  For a tiling $\mathcal{T}$ of finite local complexity, the $\mathcal{T}$-equivariant cohomology $H^*(\Omega_\mathcal{T};\mathbb{R})$ is isomorphic to the \v{C}ech cohomology $\check{H}^*(\Omega_\mathcal{T};\mathbb{R})$.
\end{theorem}
\subsection{Generators in $H^d(\Omega_x;\mathbb{R})$}
\label{subsec:generators}
Recall that for a type H family $\mathcal{F}$ there is a collection $\{\Gamma_i\}$ of AP complexes. Each AP complex $\Gamma_i$ has a CW-structure where the $d$-cells correspond to the image of the collared tiles in the projection giving the AP complex for any $\mathcal{T}\in \Omega_x$ and any $x\in U_i$. Denote by $P_1^i,\dots, P_{c(i)}^i$ the different patches corresponding to collared tiles of repetitive tilings in $\Omega_x$, $x\in U_i$. So for $x\in U_i$ and $j\in \{1,\dots, c(i)\}$, $P_j^i$ is a patch of any repetitive $\mathcal{T}\in \Omega_x$. Each patch has a distinguished point in its interior: since the $P_j^i$ are collared tiles they are of the form $P_j^i = \mathcal{T}'(t')$ for some tile $t'\in\mathcal{T}'$. As such, since $t'$ is a copy of a prototile $A_\iota$ and $A_\iota$ contains the origin, then the distringuished point in $P_j^i = \mathcal{T}'(t')$ corresponds to the point in $t'$ which is identified with the origin in $A_\iota$. This is independent of which tiling $x\in U_i$ and $\mathcal{T}'\in\Omega_x$ we use.

Let $r_{i,j}$ be the injectivity radius of $P^i_j$, $r_*$ be the minimum of all such injectivity radii and let $x\in U_i\subset \Sigma_N$. For a tiling $\mathcal{T}'\in\mho_x$, let $\Lambda_{\mathcal{T}'}^{i,j}$ be the set of vectors $\tau$ in $\mathbb{R}^d$ such that $\varphi_\tau(\mathcal{T}')\in\mho_x$ and $ \varphi_\tau(\mathcal{T}')$ contains the patch $P_j^i$ at the origin with its distinguished point exactly at the origin. In this manner, we construct $c(i)$ sets of vectors $\{\Lambda_{\mathcal{T}'}^{i,1},\dots, \Lambda_{\mathcal{T}'}^{i,c(i)}\}$, and use this to construct $c(i)$ forms $\eta^{i}_1, \dots, \eta^{i}_{c(i)} \in \Delta_\mathcal{T}^d$ as follows. Let $\rho$ be a positive, smooth bump function supported in a ball of radius $r_* /2$ and of integral 1 around the origin. Then
\begin{equation}
\label{eqn:generators}
\eta^i_j := \star\left(\sum_{x\in \Lambda_{\mathcal{T}'}^{i,j}}\rho * \delta_x\right)\in \Delta_\mathcal{T}^d,
\end{equation}
where $\star:\Delta_\mathcal{T}^0\rightarrow\Delta_\mathcal{T}^d$ is the Hodge-$\star$ operator. The forms $\eta_j^i$ can be easily described as follows. The sum of convolutions in (\ref{eqn:generators}) gives a function in $\Delta_\mathcal{T}^0$ which places a copy of the bump function $\rho$ around the distringuished point of all tiles in $\mathcal{T}'$ whose collaring is the patch $P^i_j$. Multiplying that function by the volume form gives us $\eta_j^i$.

Another way to obtain the forms $\eta_j^i$ is as follows. For $\mathcal{T}\in \Omega_x$ and $x\in U_i$, there is a projection map $\pi_{\mathcal{T}}:\mathbb{R}^d\rightarrow AP(\Omega_x)$. Then for each $j\in \{1,\dots, c(i)\}$ there exists a function $f_j^i$ obtained by placing the bump function $\rho$ on the cell of $AP(\Omega_x)$ corresponding to $P_j^i$ in such a way that $ \eta_j^i = \star \pi_{\mathcal{T}}^* f_j^i$. More generally, every form in $\Delta_{\mathcal{T}}^k$ is obtained from pulling back smooth $k$-forms on $AP^n(\Omega_x)$ from the canonical map $\pi_{\mathcal{T},n}:\mathbb{R}^d\rightarrow AP^n(\Omega_x)$ for some $n$. More specifically, by \cite[Theorem 2]{Sadun:PE} we have that
\begin{equation}
\label{eqn:pullbackForms}
\Delta_\mathcal{T}^k = \bigcup_{n\geq0} \pi_{\mathcal{T},n}^*(\Lambda^k(AP^n(\Omega_x))),
  \end{equation}
where $\Lambda^k(AP^n(\Omega_x))$ denotes the smooth $k$-forms on $AP^n(\Omega_x)$ (considering $AP^{n}(\Omega_{x})$ as a branched manifold, in the sense of \cite{Sadun:PE}).

Let $C_{k}AP(\Omega_x)$ denote the group of degree $k$ cellular cochains of $AP(\Omega_x)$. The set of $d$-cells of $AP(\Omega_x)$ is given by $c(i)$ cells (corresponding to the collared tiles $P_j^i$), so $C^{d}AP(\Omega_{x}) = \mathrm{Hom}(C_{d}AP(\Omega_x),\mathbb{R})$ is generated by the $c(i)$ forms $\star f_j^i$ dual to the collared patches $P_j^i$ and the pairing is obtained by integration over $AP(\Omega_x)$. As such, $H^d(AP(\Omega_x);\mathbb{R})$ is generated by the restriction of these forms to the kernel of the boundary map $\partial_d:C_{d}AP(\Omega_x)\rightarrow C_{d-1}AP(\Omega_x)$. In particular, $H^d(AP(\Omega_x);\mathbb{R})$ is generated by linear combinations of classes represented by the forms $\star f_j^i$.

Denote by $\gamma_{(n,m)}:AP(\Omega_{\sigma^n(x)})\rightarrow AP(\Omega_{\sigma^m(x)})$ the maps from the inverse system in (\ref{eqn:invLim}). It follows that, since the maps in (\ref{eqn:rescale}) are homeomorphisms, there exist maps $\bar{\gamma}_{(n,m)}:AP^n(\Omega_x)\rightarrow AP^m(\Omega_x)$ such that
$$\gamma_{(n,m)}\circ r_{n,x} = r_{m,x}\circ \bar{\gamma}_{(n,m)}.$$
\begin{proposition}
  \label{prop:generating}
Let $\mathcal{F} = (\mathcal{F}_1,\dots,\mathcal{F}_N)$ be a type H family and $\mu$ a minimal, ergodic $\sigma$-invariant probability measure. For $\mu$-almost every $x$, if $x\in U_i\subset \Sigma_N$, for any $\mathcal{T}'\in \Omega_x$ the forms $\eta_1^i,\dots, \eta_{c(i)}^{i}$ defined in (\ref{eqn:generators}) compose a generating set for $H^d(\Omega_x;\mathbb{R}^d)$. In other words, given $\mathcal{T}' \in \Omega_{x}$, any class $[\eta] \in H^{d}(\Omega_{x};\mathbb{R})$ is in the span of the set $\{[\eta_{j}^{i}]\} \subset H^{d}(\Omega_{x};\mathbb{R})$.

\end{proposition}
\begin{proof}
Let $x\in\Sigma_N$ be a typical point which satisfies the conclusion of Proposition \ref{prop:ES} and pick any $\mathcal{T}\in \Omega_x$.  Let $[\eta]\in H^d(\Omega_x;\mathbb{R})$, so $\eta\in\Delta^d_\mathcal{T}$. By (\ref{eqn:pullbackForms}), there exists a $k\geq 0$ such that $\eta = \pi_{\mathcal{T}, k}^* \omega$ for some $\omega\in \Lambda^d(AP^k(\Omega_x))$. Therefore $[\omega]$ denotes a class in $H^d(AP^k(\Omega_x);\mathbb{R})$ and through the homeomorphism $r_{k,x}$ we obtain the a class $[\bar{\omega}] = (r^{-1}_{k,x})^*[\omega]$ in $H^d(AP(\Omega_{\sigma^k(x)});\mathbb{R})$ represented by $\bar{\omega}\in \Lambda^d(AP(\Omega_{\sigma^k(x)}))$.

  It follows from Theorem \ref{thm:APcohomology} and Proposition \ref{prop:ES} that there exist $k_0$ and $[\omega_0]\in H^d(AP(\Omega_x);\mathbb{R})$ such that $\gamma_{k_0}^*\cdots\gamma_1^*[\omega_0] = \gamma_{k_0}^*\cdots \gamma_{k+1}^*[\bar{\omega}]$ in $H^d(AP(\Omega_{\sigma^{k_0}(x)}))$.
  As such, we have that
  \begin{equation}
    \label{eqn:equality}
 \gamma_{k_0}^*\cdots\gamma_1^*\omega_0 = \gamma_{k_0}^*\cdots \gamma_{k+1}^*\bar{\omega} + d\omega_1
    \end{equation}
in $\Lambda^d(AP(\Omega_{\sigma^{k_0}(x)}))$, for some $\omega_1\in \Lambda^{d-1}(AP(\Omega_{\sigma^{k_0}(x)}))$. Considering the commutative diagram
\[  \begin{tikzcd}
    AP(\Omega_x)   & AP(\Omega_{\sigma^k(x)})\arrow[swap]{l}{\gamma_{(k,1)}}  & AP(\Omega_{\sigma^{k_0}(x)})  \arrow[swap]{l}{\gamma_{(k_0,k)}}   \\
    AP(\Omega_x) \arrow{u}{r_{1,x}} & AP^k(\Omega_x) \arrow[swap]{l}{\bar{\gamma}_{(k,1)}}   \arrow{u}{r_{k,x}}      & AP^{k_0}(\Omega_x)  \arrow[swap]{l}{\bar{\gamma}_{(k_0,k)}}  \arrow{u}{r_{k_0,x}} \\
    \mathbb{R}^d \arrow{u}{\pi_{\mathcal{T},1}} & \mathbb{R}^d \arrow{l}{\mbox{id}}   \arrow{u}{\pi_{\mathcal{T},k}}      & \mathbb{R}^d  \arrow{l}{\mbox{id}}  \arrow{u}{\pi_{\mathcal{T},k_0}} \\
    \end{tikzcd}
\]
we obtain the dual diagram
\begin{equation}
  \label{eqn:dualDiagram}
\begin{tikzcd}
    \Lambda^d(AP(\Omega_x)) \arrow{d}{r_{1,x}^*} \arrow{r}{\gamma^*_{(k,1)}}  & \Lambda^d(AP(\Omega_{\sigma^k(x)})) \arrow{d}{r_{k,x}^*} \arrow{r}{\gamma^*_{(k_0,k)}}  & \Lambda^d(AP(\Omega_{\sigma^{k_0}(x)}))    \arrow{d}{r_{k_0,x}^*}  \\
    \Lambda^d(AP(\Omega_x)) \arrow{d}{\pi^*_{\mathcal{T},1}} \arrow{r}{\bar{\gamma}^*_{(k,1)}}  & \Lambda^d(AP^k(\Omega_x))\arrow{d}{\pi^*_{\mathcal{T},k}} \arrow{r}{\bar{\gamma}^*_{(k_0,k)}}  & AP^{k_0}(\Omega_x)  \arrow{d}{\pi^*_{\mathcal{T},k_0}} \\
  \pi^*_{\mathcal{T},1} \Lambda^d(AP(\Omega_x))   \arrow{r}{i_{1,k}} &  \pi^*_{\mathcal{T},k} \Lambda^d(AP^k(\Omega_x))  \arrow{r}{i_{k,k_0}}    & \pi^*_{\mathcal{T},k_0} \Lambda^d(AP^{k_0}(\Omega_x))   \\
    \end{tikzcd}
\end{equation}
where $i_{m,n}$ denotes the natural inclusion of one set of $\mathcal{T}$-equivariant forms into a larger set. Now, by (\ref{eqn:equality}), we have that
$$d\omega_1 = \gamma^*_{(k_0,1)}\omega_0 - \gamma^*_{(k_0,k)}\bar{\omega},$$
so using (\ref{eqn:dualDiagram}),
\begin{equation}
  \label{eqn:diagChase}
  \begin{split}
    \pi^*_{\mathcal{T},k_0}r^*_{k_0,x} d\omega_1 &=\pi^*_{\mathcal{T},k_0}r^*_{k_0,x} \gamma^*_{(k_0,1)}\omega_0 - \pi^*_{\mathcal{T},k_0}r^*_{k_0,x}\gamma^*_{(k_0,k)}\bar{\omega} =\pi^*_{\mathcal{T},k_0}r^*_{k_0,x} \gamma^*_{(k_0,1)}\omega_0 - i_{k,k_0}\pi^*_{\mathcal{T},k}r^*_{k,x}\bar{\omega}\\
    &=i_{1,k_0}\pi^*_{\mathcal{T},1}r^*_{1,x} \omega_0 - i_{k,k_0}\pi^*_{\mathcal{T},k_0}\omega =i_{1,k_0}\pi^*_{\mathcal{T},1}r^*_{1,x} \omega_0 - i_{k,k_0}\pi^*_{\mathcal{T},k_0}\omega \\
    &=i_{1,k_0}\pi^*_{\mathcal{T},1} \omega_0 - i_{k,k_0}\eta =i_{1,k_0}\pi^*_{\mathcal{T},1} \omega_0 - i_{k,k_0}\eta = \pi^*_{\mathcal{T},1} \omega_0 - \eta.
  \end{split}
\end{equation}
The pullback $\pi^*_{\mathcal{T},k_0}r^*_{k_0,x} d\omega_1$ of the exact form $d\omega_1$ is exact, so we denote it by $d\omega_2$. Finally, the class $[\omega_0]$ is represented by a linear combination of the forms $\star f_j^i$:
\begin{equation}
  \label{eqn:omega0}
  \omega_0 = \sum_{j=1}^{c(i)} \beta_j(\omega_0)\star f_j^i,
  \end{equation}
where $\beta_j(\omega_0)\in\mathbb{R}$. So we have from (\ref{eqn:diagChase}) that
$$d\omega_2 = \pi^*_{\mathcal{T},1} \omega_0 - \eta =  \sum_{j=1}^{c(i)} \beta_j(\omega_0)\pi^*_{\mathcal{T},1}\star f_j^i -\eta = \sum_{j=1}^{c(i)} \beta_j(\omega_0)\eta_j^i -\eta,$$
which concludes the proof.
\end{proof}
\subsection{A norm on $H^d(\Omega_x;\mathbb{R})$}
\label{subsec:norm}
Recall from \S \ref{subsec:AP} that given a type H family $\mathcal{F} = \{\mathcal{F}_1,\dots,\mathcal{F}_n\}$, there exists a partition $\{U_i\}$ of $\Sigma_n$ and CW-complexes $\{\Gamma_i\}$ such that $AP(\Omega_x) = \Gamma_i$ if $x\in U_i$. In order to endow $H^d(\Omega_x;\mathbb{R})$ with a norm, we first equip each $H^d(AP(\Omega_x);\mathbb{R})$ with a norm. Since each CW-complex $\Gamma_i$ is a finite complex, $H^d(\Gamma_i;\mathbb{R})$ is finite dimensional for each $i$ and so we can endow it with its natural $L^p$ norm $\|\cdot\|_p$. The following is a consequence of Proposition \ref{prop:ES}.

\begin{corollary}
\label{cor:cohomIdent}
  Let $\mathcal{F} = \{\mathcal{F}_1,\dots,\mathcal{F}_n\}$ be a type H family and $\mu$ a minimal ergodic $\sigma$-invariant probability measure. Then for $\mu$-almost every $x$, there is a subspace $ES_x^*\subset H^*(AP(\Omega_x);\mathbb{R})$ such that $H^*(\Omega_x;\mathbb{R})$ is naturally isomorphic to $ES_x^*$.
\end{corollary}
By naturally isomorphic, we mean that each class in the direct limit presentation of $H^*(\Omega_x;\mathbb{R})$ has a representative in $H^*(AP(\Omega_x);\mathbb{R})$. See the proof of Proposition \ref{prop:ES} for details. By the identification of $H^*(\Omega_x;\mathbb{R})$ with $ES_x^*\subset H^*(\Gamma_i;\mathbb{R})$ given by Corollary \ref{cor:cohomIdent}, for $x\in U_i$ we can now endow $H^*(\Omega_x;\mathbb{R})$ with a norm: the restriction of the $L^p$ norm in $H^*(\Gamma_i;\mathbb{R}) = H^*(AP(\Omega_x);\mathbb{R})$ to the subspace $ES_x^*\subset H^*(AP(\Omega_x);\mathbb{R})$.\\

Let us now define a specific norm which will be useful for the bounds needed in \S \ref{sec:erg} to prove the main theorem. Let $\{C_k^i\}$ be the cellular chain complex of $\Gamma_i$. Recall that since $C_d^i$ is generated by the $d$-faces $\{c_k\}$ of $\Gamma_i$, $\mathrm{Hom}(C_d^i,\mathbb{R})$ is generated by the dual $c(i)$ functions $\{f_1^i,\dots, f_{c(i)}^i\}$ introduced after (\ref{eqn:generators}), where the pairing comes by $f_j^i(c) = \int_c \star f_j^i$ for any $c\in C_d^i$. This pairing gives an $L^\infty$-type of norm on $\mathrm{Hom}(C_d^i,\mathbb{R})$ by
\begin{equation}
  \label{eqn:norm1}
\|f\|' = \max_{k}\left|f(c_k)\right| = \max_{k}\left|\int_{c_k} \star f \right|.
  \end{equation}
Now, for $x\in U_i$, since
$$H^d(AP(\Omega_x);\mathbb{R}) = \mathrm{Hom}\left.\left(C_d^i,\mathbb{R}\right)\right|_{\mathrm{ker}(\partial_d:C_d^i\rightarrow C_{d-1}^i)} = \left.\left\langle f_1^i,\dots, f_{c(i)}^i\right\rangle\right|_{\mathrm{ker}(\partial_d:C_d^i\rightarrow C_{d-1}^i)}$$
the space $H^d(AP(\Omega_x)$ is generated by a linear combination of the functions $\{f_j^i\}$. Furthermore, since by Corollary \ref{cor:cohomIdent} we have that $H^*(\Omega_x;\mathbb{R})\subset H^*(AP(\Omega_x))$, $H^*(\Omega_x;\mathbb{R})$ is also generated by a linear combination of the functions $\{f_j^i\}$. Thus the norm (\ref{eqn:norm1}) restricts to an $L^\infty$ norm on $H^d(\Omega_x;\mathbb{R})$ as follows. For a representative $\eta = \sum_{j=1}^{c(i)}\beta_j(\eta)\star f_j^i$ of a class in $H^d(\Omega_x;\mathbb{R})$ then
\begin{equation}
  \label{eqn:norm2}
  \|[\eta]\| = \max_k \left|\int_{c_k} \eta \right| = \max_k \left| \int_{c_k} \sum_{j=1}^{c(i)} \beta_j(\eta)\star f_j^i\right| = \max_k \left|  \sum_{j=1}^{c(i)} \beta_j(\eta)\int_{c_k} \star f_j^i\right| = \max_k \left| \beta_k(\eta)  \right|\geq 0
  \end{equation}
defines a norm on $H^d(\Omega_x;\mathbb{R})$. The following proposition shows this norm may be written in a slightly different way.
\begin{proposition}
\label{prop:norm}
Let $\mathcal{F} = (\mathcal{F}_1,\dots,\mathcal{F}_N)$ be a type H family and $\mu$ a minimal, ergodic $\sigma$-invariant probability measure. For $\mu$-almost every $x$, the function $\|\cdot\|:H^d(\Omega_x;\mathbb{R})\rightarrow\mathbb{R}$ defined by
\begin{equation}
  \label{eqn:norm3}
\|c\| = \max_{t\in \mathcal{T}}\left|\int_{t}\eta\right|
\end{equation}
where $\eta\in \Delta_\mathcal{T}^d$, $\mathcal{T}\in\Omega_x$, is the representative of the class $c\in H^d(\Omega_x;\mathbb{R})$ of the type given by Proposition \ref{prop:generating}, gives a norm on $H^d(\Omega_x;\mathbb{R})$.
\end{proposition}
\begin{proof}
  For such a minimal, ergodic $\sigma$-invariant probability measure $\mu$, let $x\in U_i\subset  \Sigma_N$ be such that the conclusion of Proposition \ref{prop:generating} holds. Let $c\in H^d(\Omega_x;\mathbb{R})$ be a class. By Proposition \ref{prop:generating} and its proof, for any $\mathcal{T}\in \Omega_x$, we have $[c] = [\eta_{c}]$ where
  $$\eta_c = \sum_{j=1}^{c(i)} \beta_j(c)\pi_{\mathcal{T},1}^*\eta_j^i\in \Delta_\mathcal{T}^d.$$
  So we have that
  \begin{equation*}
    \begin{split}
  \max_{t\in\mathcal{T}}\left|\int_t \eta_c\right| &= \max_{t\in\mathcal{T}}\left|\int_t  \sum_{j=1}^{c(i)} \beta_j(c)\pi_{\mathcal{T},1}^*\eta_j^i \right| = \max_{t\in\mathcal{T}}\left|  \sum_{j=1}^{c(i)} \beta_j(c)\int_t\pi_{\mathcal{T},1}^*\eta_j^i \right|= \max_{t\in\mathcal{T}}\left|  \sum_{j=1}^{c(i)} \beta_j(c)\int_{\pi_{\mathcal{T},1} t}\star f_j^i \right| \\
  &= \max_{k}\left| \beta_j(\eta) \int_{c_k} \star f_j^i \right| = \max_k \left|\beta_k(\eta)\right|.
  \end{split}
  \end{equation*}
 Comparing with (\ref{eqn:norm2}), the result follows.
\end{proof}

\section{The cohomology bundle}
\label{sec:bundle}
\begin{definition}
  Given a type H family $\mathcal{F} = \{\mathcal{F}_1,\dots, \mathcal{F}_N\}$, the \textbf{cohomology bundle} of this family is the trivial bundle $\mathcal{H}_{\mathcal{F}} := \Sigma_N\times H^d(\Omega)$ over $\Sigma_N$ having as fiber over $x\in\Sigma_N$ the vector space $H^d(\Omega_x;\mathbb{R})$.
\end{definition}
We endow each fiber $H^d(\Omega)$ of $\mathcal{H}_\mathcal{F}$ with the norm defined in \S \ref{subsec:norm}, and we write $\|\cdot \|_x$ for the norm on the fiber $H^d(\Omega_x;\mathbb{R})$. Since the bundle $\mathcal{H}_\mathcal{F}$ is over a Cantor set, the notion of a connection does not make sense right away. However, given a minimal, ergodic $\sigma$-invariant probability measure $\mu$, Corollary \ref{cor:cohomIdent} gives a way to compare nearby fibers for $\mu$-almost every fiber.
Since $\mathcal{F}$ is uniformly affine scaling, given $x\in\Sigma_N$, we denote by $A_{x_1} = \theta^{-1}_{x_1}\cdot \mathrm{Id} $ the expanding matrix associated with the maps in $\mathcal{F}_{x_1}$.
\begin{lemma}
  \label{lem:Abijection}
  Let $\mathcal{F}$ be a type H family and $\Omega_x$ be the tiling space for $x\in\Sigma_N$. Let $[\omega]\in H^d(\Omega_{\sigma(x)};\mathbb{R})$ where $\omega \in \Delta^{d}_{\mathcal{T}}$ for $\mathcal{T} \in \Omega_{\sigma(x)}$. Then $\Phi_x^*[\omega]\in H^d(\Omega_{x};\mathbb{R})$ is represented by the pattern-equivariant form $(A_{x_{1}}^{-1})^*\omega\in \Delta^d_{\Phi_x(\mathcal{T})}$.
\end{lemma}
\begin{proof}
We trace back the action through the isomorphisms $i_\mathcal{T}:C^\infty_{tlc}(\Omega_x)\rightarrow \Delta^0_\mathcal{T}$ and $i_{\Phi_x(\mathcal{T})}:C^\infty_{tlc}(\Omega_{\sigma(x)})\rightarrow \Delta^0_{\Phi_x(\mathcal{T})}$ from Theorem \ref{thm:algIso}:
  \begin{equation*}
    \begin{split}
      \left(i_\mathcal{T}\Phi_x^* i_{\Phi_x(\mathcal{T})}^{-1}f\right)(t)&= \left(\Phi_x^* i_{\Phi_x(\mathcal{T})}^{-1} f \right) (\varphi_t(\mathcal{T})) = i^{-1}_{\Phi_x(\mathcal{T})}f(\Phi_x\circ\varphi_t(\mathcal{T})) \\
      &= i^{-1}_{\Phi_x(\mathcal{T})} f (\varphi_{A_{x_{1}}^{-1}t}\circ \Phi_x(\mathcal{T})) = \left(i_{\Phi_x(\mathcal{T})}i^{-1}_{\Phi_x(\mathcal{T})}f\right)(A_{x_{1}}^{-1}t) = f(A_{x_{1}}^{-1}t),
    \end{split}
  \end{equation*}
  where we used the conjugacy from Proposition \ref{prop:conj} in the third equality.
\end{proof}

\begin{definition}
  The \textbf{renormalization cocycle} is the map $\varsigma:\mathcal{H}_{\mathcal{F}}\rightarrow \mathcal{H}_{\mathcal{F}}$ defined as $(x,[\eta])\mapsto (\sigma(x),\Theta_x[\eta])$, where $\Theta_x := (\Phi_x^{-1})^*:H^d(\Omega_x;\mathbb{R})\rightarrow H^d(\Omega_{\sigma(x)};\mathbb{R})$.
\end{definition}
We will denote products as $A_{(n)_x} := A_{x_{n}}\dots A_{x_{1}}$. We now appeal to Oseledets theorem. In what follows $\|\cdot\|$ denotes the operator norm, and $\log^+(x) = \max\{0,\log(x)\}$.
\begin{theorem}[Oseledets theorem]
  \label{thm:oseledets}
    Let $\mathcal{F}_1,\dots, \mathcal{F}_N$ be a type H family. Let $\mu$ be an minimal ergodic, $\sigma$-invariant probability measure on $\Sigma_N$. Suppose furthermore that
    $$\int_{\Sigma_N}\log^+ \|\Theta_x\|\, d\mu < \infty.$$
    Then there exist Lyapunov exponents $\lambda_1\geq\lambda_2\geq \cdots\geq\lambda_{r_\mu}$ such that for $\mu$-almost all $x\in\Sigma_N$ there is a $\varsigma$-invariant measurable splitting of $\mathcal{H}_\mathcal{F}$
    \begin{equation}
      \label{eqn:splitting}
H^d(\Omega_x;\mathbb{R}) = \bigoplus_{i=1}^{r_\mu} E_i(x)
    \end{equation}
such that for any $[\eta]\in E_i(x) \backslash \{0\}$,
\begin{equation}
  \label{eqn:LyapExp}
  \lim_{n\rightarrow\infty}\frac{1}{n}\log\|\Theta_{\sigma^{n-1}(x)}\circ\cdots\circ \Theta_{x}[\eta]  \| = \lambda_i.
  \end{equation}
\end{theorem}
We note that the condition
$$\int_{\Sigma_N}\log^+ \|\Theta_x\|\, d\mu < \infty$$
holds in particular for the renormalization cocycle $\Theta_x = (\Phi_x^{-1})^*:H^d(\Omega_x;\mathbb{R})\rightarrow H^d(\Omega_{\sigma(x)};\mathbb{R})$; indeed, the cocycle takes finitely many values, over the partition $U_{i}$.
  \begin{definition}
The \textbf{rapidly expanding subspace} $E^+_x$ is the subspace spanned by the collection of Oseledets subspaces $E_i(x)$ in (\ref{eqn:splitting}) with Lyapunov exponent $\lambda_i$ in (\ref{eqn:LyapExp}) satisfying $\lambda_i \geq \frac{d-1}{d}\lambda_1$.
\end{definition}
\begin{definition}
  Let $\mathcal{F}$ be a type H family and $\mu$ be a minimal, $\sigma$-invariant probability measure on $\Sigma_N$. For the Lyapunov spectrum $\lambda_1\geq \lambda_2\geq \dots\geq \lambda_{r_\mu}$ of $\mu$, the \textbf{normalized Lyapunov exponents} $\nu_i$ are given by
  $$\nu_i = d\frac{\lambda_i}{\lambda_1}$$
  for all $i=1,\dots, r_\mu$. Note that $E_i(x)\subset E^+_x$ if and only if $\nu_i\geq d-1$.
\end{definition}
To compactify notation, we denote by
$$\Phi_x^{(n)} := \Phi_{\sigma^{n-1}(x)}\circ\cdots\circ\Phi_x$$
the composition of the maps from Proposition \ref{prop:conj} along orbits of $x$.
\begin{lemma}
  \label{lem:cocycleForms}
  Let $\mathcal{F} = (\mathcal{F}_1,\dots,\mathcal{F}_N)$ be a type H family and $\mu$ a minimal, ergodic $\sigma$-invariant probability measure.  For $\mu$-almost every $x$ and any $\mathcal{T}\in\Omega_x$ we have that
  $$\left\|\Theta_{\sigma^{n-1}(x)}\circ\cdots\circ \Theta_x[\eta]\right\|=  \max_{t\in\Phi_x^{(n)}(\mathcal{T})} \left|\int_{A_{(n)_x}t} \eta \right|,$$
  where the norm on the left is the one from (\ref{eqn:norm2}) (or, equivalently, from (\ref{eqn:norm3})) and $\eta\in \Delta_\mathcal{T}^d$ is the representative of the class $[\eta]$ given in Proposition \ref{prop:generating}.
\end{lemma}
\begin{proof}
  For such a minimal, ergodic $\sigma$-invariant probability measure $\mu$, let $x\in U_i\subset  \Sigma_N$ be such that the conclusion of Proposition \ref{prop:generating} holds, and pick $\mathcal{T}\in\Omega_x$. By Lemma \ref{lem:Abijection}, for any $n>0$ there is a $\mathcal{T}_n  = \Phi_x^n(\mathcal{T})\in\Omega_{\sigma^n(x)}$ and $ \varepsilon_n>0$ such that $A^*_{(n)_x}\eta\in \Delta_{\mathcal{T}_n}^d$ is a representative of the class $\Theta_{\sigma^{n-1}(x)}\circ\cdots\circ \Theta_x[\eta]$ and, moreover, $\mathrm{supp}\, A^*_{(n)_x}\eta \cap \, N_{\varepsilon_n}(\partial t) = \varnothing $ for all $t\in \mathcal{T}_n$, where $N_\varepsilon(S)$ denotes the $\varepsilon$-neighborhood of the set $S$. In other words, the form $A^*_{(n)_x}\eta$ is supported way from the union of the boundaries of the tiles of $\mathcal{T}_n$. We note that $\mathcal{T}_n = \Phi_x^n(\mathcal{T})$ and $\mathcal{T}$ are related in a very special way: by Proposition \ref{prop:conj}, the tiling $\mathcal{T}$ is obtained from the tiling  $\mathcal{T}_n$ by performing $n$ substitutions and inflations according to the substitution rules $\mathcal{F}_{x_{n}}, \mathcal{F}_{x_{n-1}},\dots , \mathcal{F}_{x_{1}}$. This is why the support of $A^*_{(n)_x}\eta$ is contained in the interior of the tiles of $\mathcal{T}_n$.

  Now, applying the construction of Proposition \ref{prop:generating} to the class $[A^*_{(n)_x}\eta] $ we obtain a form
$$\eta' = \sum_{j=1}^{c(i')}\beta_j(\eta')\eta_j^{i'},$$
  where the forms $\eta_j^{i'}\in \Delta_{\mathcal{T}_n}^d$ come from (\ref{eqn:generators}). Since they both represent the same class, we have that $ \eta_j^{i'}- A^*_{(n)_x}\eta = d\omega_n$. Since both $A^*_{(n)_x}\eta$ and $\eta_j^{i'}$ are supported away from the union of the boundaries of all tiles $t\in\mathcal{T}_n$, so is $d\omega_n$. Thus we have
  \begin{equation*}
    \begin{split}
      \|\Theta_{\sigma^{n-1}(x)}\circ\cdots\circ \Theta_x[\eta]\|&=  \max_{t\in \mathcal{T}_n} \left|\int_{t} \eta' \right| = \max_{t\in \mathcal{T}_n} \left|\int_{t} A^*_{(n)_x}\eta + d\omega_n \right| = \max_{t\in \mathcal{T}_n} \left|\int_{t} A^*_{(n)_x}\eta + \int_{\partial t}\omega_n \right| \\
  &= \max_{t\in \mathcal{T}_n} \left|\int_{t} A^*_{(n)_x}\eta \right|= \max_{t\in \mathcal{T}_n} \left|\int_{A_{(n)_x}t} \eta \right|
    \end{split}
  \end{equation*}
  where the first equality follows from (\ref{eqn:norm3}) in Proposition \ref{prop:norm}.
  \end{proof}

\section{Ergodic integrals}
\label{sec:erg}
Given that we will study averages of functions, we need to define the types of averaging sets which will be used. Given a compact set $B$ with non-empty interior, denote by $T\cdot B$ the one-parameter family of sets obtained from $B$ through
$$T\cdot B := \mathrm{diag}(T)\cdot B.$$
As such, we have that $\mathrm{Vol}(T\cdot B) = \mathrm{Vol}(B) T^d$.

Recall by Lemma \ref{lem:Abijection} that the renormalization map acts on forms through scaling matrices. More precisely, for $\mathcal{T}\in \Omega_x$ if $\eta\in \Delta_\mathcal{T}^d$ represents a class in $H^d(\Omega_x;\mathbb{R})$, then $\Theta_x[\eta]$ is represented by $A_{x_{1}}^*\eta\in \Delta^d_{\Phi_x(\mathcal{T})}$, where $A_{x_{1}}$ is the diagonal matrix with all entries $\theta_{x_{1}}^{-1}$. To reduce the amount of tedious notation, we denote the renormalization cocyle actions by 
\begin{equation}
  \label{eqn:cocycles}
  \Theta_{x}^{(n)} := \Theta_{\sigma^{n-1}(x)}\circ\cdots\circ \Theta_x, \hspace{.25in}    A_{(n)_x}:= A_{x_{n}}\cdots A_{x_{1}},  \hspace{.25in} \mbox{ and }\hspace{.25in} \theta_{(n)_x} := \theta_{x_{1}}\cdots \theta_{x_{n}}.
\end{equation}
Finally, for a type H family $\mathcal{F} = \{\mathcal{F}_1,\dots, \mathcal{F}_N\}$ and $x\in\Sigma_N$, let $T_{x,n} = \theta_{(n)_x}$. As such, we have
\begin{equation}
  \label{eqn:TimeScaling}
 A_{(n)_x} = \mathrm{diag}(T_{x,n}).
\end{equation}
We now make a basic observation about the leading Lyapunov exponent $\lambda_1$. Let $x\in \Sigma_N$ be an Oseledets regular point for the renormalization cocycle for some ergodic, minimal $\sigma$-invariant probability measure $\mu$. By Lemma \ref{lem:Abijection},
if the volume form $\star 1$ represents the class $[\eta_1]\in H^d(\Omega_x;\mathbb{R})$, then $A_{(n)_x}^*(\star 1) = \mathrm{det}(A_{(n)_x})(\star 1) = \theta_{(n)_x}^{-d}(\star 1)=(\theta_{x_{1}}\dots \theta_{x_{n}})^{-d}(\star 1)$ represents the class $\Theta_x^{(n)}[\eta_1]\in H^d(\Omega_{\sigma^{1}(x)};\mathbb{R})$. Thus, Oseledets theorem establishes that
\begin{equation}
  \label{eqn:Oseledets3}
\lim_{n\rightarrow \infty}\frac{\log (\theta_{x_{1}}\dots \theta_{x_{n}})^{-d}}{n} = d\lim_{n\rightarrow \infty} \frac{1}{n}\sum_{i=1}^n -\log \theta_{x_{i}}= \lambda_1.
\end{equation}
Using (\ref{eqn:TimeScaling}), it follows that
\begin{equation}
\label{eqn:returnTimes}
  \lim_{n\rightarrow\infty}\frac{\log T_{n,x}}{n} = \lim_{n\rightarrow\infty}\frac{\log \theta_{(n)_x}}{n} = \lim_{n\rightarrow \infty} \frac{1}{n}\sum_{i=1}^n -\log \theta_{x_{i}}= \frac{\lambda_1}{d}.
\end{equation}
\subsection{Upper bound}
\begin{lemma}
\label{lem:packing}
  Let $\mathcal{F}$ be a type H family and $\mu$ a minimal, $\sigma$-invariant ergodic probability measure on $\Sigma_N$. For $B$ a Lipschitz domain with non-empty interior and tiling $\mathcal{T}\in\Omega_x$ and $T>0$ there exists an integer $n = n(T,B)$ and a decomposition
  \begin{equation}
    \label{eqn:decomp3}
    \mathcal{O}_\mathcal{T}^-(T\cdot B) = \bigcup_{i = 0}^n\bigcup_{j=1}^M \bigcup_{k=1}^{\kappa_j^{(i)}}t^{(i)}_{j,k},
  \end{equation}
where $t^{(i)}_{j,k}$ is a level-$i$ supertile of the tiling $\mathcal{T}$ of type $j$, such that
  \begin{enumerate}
  \item $\kappa_j^{(n)}\neq 0$ for some $j$ and $\mathrm{Vol}(T\cdot B)\leq K_1 \theta_{(n)_x}^{-d}$,
    \item $\sum_{j=1}^M \kappa_j^{(i)} \leq K_2 \mathrm{Vol}(\partial T\cdot B)  \theta_{(i)_x}^{d-1}$ for $i = 0,\dots, n-1$
  \end{enumerate}
  for some $K_1, K_2$ which depend only on $\mathcal{F}$ and $B$.
\end{lemma}
Before proving the lemma, we establish an inequality related to efficient hierarchical packings of Lipshitz domains by supertiles of different orders. These types of estimates have been done elsewhere before, see for example \cite[Page 769]{BufetovSolomyak}.

For a tiling $\mathcal{T}\in \Omega_x$ and Lipshitz domain $B$, let $\mathcal{T}^{(k)}_B$ denote the set of all supertiles of order $k$ completely contained in $B$. Further, let $\mathcal{R}^{(k)}_B$ be the set of supertiles of order $k$ which belong to supertiles of order $k+1$ which \textbf{are not} completely contained in $B$. As such, the supertiles in $\mathcal{R}^{(k)}_B$ are contained in supertiles of order $k+1$ which intersect $\partial B$. Let $d_-, d_+$ be, respectively, the smallest and largest diameters of the prototiles $t_1,\dots, t_M$. Since the largest diameter of level $k+1$ supertile is $d_+ \theta_{(k)_x}^{-1}$ we have that $\mathcal{R}^{(k)}_B \subset \partial_{d_+\theta_{(k+1)_x}^{-1}}(B)$.

Let $a_-$ be the smallest of all of the volumes of the prototiles $t_1,\dots, t_M$. Then the volume of an order $k$ supertile in $\Omega_x$ is at least $a_- \theta_{(k)_x}^{-d}$ and we have that
\begin{equation}
  \label{eqn:packing}
  |\mathcal{R}^{(k)}_B| \leq \frac{\mathrm{Vol}(\partial_{d_+\theta_{(k+1)_x}^{-1}}(B))}{a_- \theta_{(k)_x}^{-d} }.
\end{equation}
\begin{proof}[Proof of Lemma \ref{lem:packing}]
  The idea here is to decompose $\mathcal{O}_\mathcal{T}^-(T\cdot B)$ into tiles of different heirarchical levels beginning from the top level $n(T,B)$ and filling it in using smaller tiles.  First we find $n = n(T,B)$, after which the first property of the decomposition will follow.

  Let $R_t>0$ be the smallest number such that in any $\mathcal{T}\in \Omega_x$, any ball of radius $R_t$ contains a tile of $\mathcal{T}$. Let $m_1\in\mathbb{N}$ be the smallest $m$ such that $B_{\theta_{max}^{-m_1}}$ contains a ball of radius $R_t$, where $\theta_{max}$ is the largest contraction constant in the family $\mathcal{F}$; thus $B_{\theta_{max}^{-m_1}}$ contains a tile of $\mathcal{T}$. For $T\geq \theta_{max}^{-m_1}$, let $n(T,B)$ be the largest $n\in\mathbb{Z}^+$ such that there exists a level-$n$ supertile $t^{(n)}$ of $\mathcal{T}$ completely contained in $T\cdot B$. So there is a finite set of level-$n$ supertiles $\{t_{j,k}^{(n)}\}$, $j = 1,\dots, M$ and $k = 1,\dots, \kappa_j^{(n)}$, where $t^{(n)}_{j,k}$ is a supertile of type $j$ such that
  $$\mathcal{T}^{(n)}_{T\cdot B}:= \bigcup_{j=1}^M\bigcup_{k=1}^{\kappa_j^{(n)}} t^{(n)}_{j,k}\subset \mathcal{O}_\mathcal{T}^-(T\cdot B),$$
where $\mathcal{T}^{(n)}_{T\cdot B}$ denotes all the supertiles of order $n$ completely contained in $T\cdot B$. So $\kappa_j^{(n)}\neq 0$ for some $j$.

  Given the definition of $R_t$, it follows that any ball of radius $\theta_{(n)_x}^{-1}\theta^{-m_1}_{max}$ contains a supertile of order $n$ for any tiling $\mathcal{T}\in\Omega_x$. Let $r_t$ be the smallest of all injectivity radii of the prototiles. Let $m_2$ be the smallest integer such that $B_{\theta_{max}^{-m_2}r_t}(y)$ contains a supertile of order $1$ for all $y\in\mathbb{R}^d$ and $\mathcal{T}\in\Omega_x$. There is a $m^-\in\mathbb{Z}$ such that $B_{\theta_{min}^{m^-}T}(y)$ contains no supertiles of order $n$ for any $y\in\mathbb{R}^d$ and $\mathcal{T}\in\Omega_x$. So we have that
  \begin{equation}
    \label{eqn:nTbounds}
    \theta_{min}^{m^-}\theta_{(n)_x}^{-1} < T < \theta_{max}^{-m_1-m_2}\theta_{(n)_x}^{-1}\hspace{.4in} \mbox{ and }\hspace{.4in}n(\theta_{max}^{-m_2} T,B) > n(T,B)
    \end{equation}
    and it follows that since $T\cdot B$ contains a supertile of order $n$ and $(\theta_{max}^{-m_2} T)\cdot B(x)$ contains a supertile of order $n+1$,
$$\mathrm{Vol}(T\cdot B) \leq \mathrm{Vol}(B)\theta_{max}^{-m_1}\theta_{(n)_x}^{-d} \theta_{max}^{-m_2}  = K_1(B,\mathcal{F})\theta_{(n)_x}^{-d},$$
from which the first property follows.

If $n(T,B)=0$, we are done. Otherwise we now look at $\mathcal{O}_\mathcal{T}^-(T\cdot B)\backslash \mathcal{T}^{(n)}_{T\cdot B}$ and look for patches corresponding to level-$n-1$ supertiles of $\mathcal{T}$ which are contained in $\mathcal{O}_\mathcal{T}^-(T\cdot B)\backslash \mathcal{T}^{(n)}_{T\cdot B}$. Let
$$\mathcal{R}^{(n-1)}_{T\cdot B}:= \bigcup_{j=1}^M\bigcup_{k=1}^{\kappa_j^{(n-1)}}t_{j,k}^{(n-1)}\subset \mathcal{O}_\mathcal{T}^-(T\cdot B)\backslash \mathcal{T}^{(n)}_{T\cdot B}$$
denote their union (it is possible that $\mathcal{R}^{(n-1)}_{T\cdot B} = \varnothing$). Proceeding recursively in this way we obtain a decomposition of $\mathcal{O}_\mathcal{T}^-(T\cdot B)$ in terms of supertiles of different orders as in (\ref{eqn:decomp3}).

We now use (\ref{eqn:packing}) to estimate the numbers $\kappa_j^{(k)}$, which are the number of supertiles of order $k$ of type $j$ used in the decomposition:
\begin{equation}
  \label{eqn:packing2}
  \begin{split}
    \sum_{j=1}^M\kappa_{j}^{(i)} &= |\mathcal{R}^{(i)}_{T\cdot B}| \leq \frac{\mathrm{Vol}(\partial_{d_+\theta_{(i+1)_x}^{-1}}(T\cdot B))}{a_- \theta_{(i)_x}^{-d} } \leq K_2 \frac{ \mathrm{Vol}(\partial T\cdot B) d_+\theta_{(i+1)_x}^{-1}  \theta_{(i)_x}^{d}}{a_-} \\
    &\leq K_2 \mathrm{Vol}(\partial T\cdot B)  \theta_{(i)_x}^{d-1} ,
  \end{split}
\end{equation}
where we implicitly used that $B$ is a Lipschitz domain in the second inequality of (\ref{eqn:packing2}) (see \cite[eq. (6)]{BufetovSolomyak}).
\end{proof}
\begin{proposition}
\label{prop:upBnd}
  Let $\mathcal{F} = \{\mathcal{F}_1,\dots, \mathcal{F}_N\}$ be a type H family, $\mu$ an ergodic, minimal $\sigma$-invariant probability measure on $\Sigma_N$, and $B\subset\mathbb{R}^d$ a compact subset with non-empty interior. For $\mu$-almost every $x$ and any $\mathcal{T}\in \Omega_x$, for $\eta_\ell\in\Delta^d_\mathcal{T}$ representing a class in $E^+_\ell(x)$ of the form given by Proposition \ref{prop:generating} we have
  $$\limsup_{T\rightarrow \infty}\frac{\log \left| \displaystyle\int_{T\cdot B} \eta_\ell  \right|}{\log T} \leq \nu_\ell,$$
  where $\nu_\ell$ is the $\ell^{th}$ normalized Lyapunov exponent of $\mu$.
\end{proposition}
\begin{proof}
  By Proposition \ref{prop:generating} we can choose the representative $\eta_\ell$ of a particular form, namely a linear combination of forms of the form (\ref{eqn:generators}), where $\rho$ is a bump function whose support is small enough that each $\eta_j^i$ in (\ref{eqn:generators}) consists of a bump function supported entirely inside the tile being collared to give $P_j^i$. We first decompose the integral into two integrals as
  \begin{equation}
    \label{eqn:twoInts}
    \left| \int_{T\cdot B} \eta_\ell \right|\leq \left| \int_{\mathcal{O}_\mathcal{T}^-(T\cdot B)} \eta_\ell \right| + \left| \int_{T\cdot B\backslash \mathcal{O}_\mathcal{T}^-(T\cdot B)} \eta_\ell \right| = I_1 + I_2.
 \end{equation}
 We begin with $I_1$. Using the decomposition given by Lemma \ref{lem:packing}, the relationship between supertiles $t^{(i)}_{j,k}$ of $\mathcal{T}$ and tiles of $\Phi_x^i(\mathcal{T})$ given by Proposition \ref{prop:conj}, and the expression for the norm in Lemma \ref{lem:cocycleForms}, by Oseledets theorem, given $\varepsilon > 0$, there exist $K_4,K_5, K_6$ such that
\begin{equation}
  \label{eqn:upBnd1}
    \begin{split}
      \left| \int_{\mathcal{O}_\mathcal{T}^-(T\cdot B)}\eta_\ell \right| &= \left| \sum_{i=1}^n\sum_{j=1}^M \sum_{k=1}^{\kappa_j^{(i)}} \int_{t^{(i)}_{j,k}}\eta_\ell \right| \leq  \sum_{i=1}^n\sum_{j=1}^M \sum_{k=1}^{\kappa_j^{(i)}} \left|\int_{t^{(i)}_{j,k}}\eta_\ell \right| \\
      &\leq  \sum_{i=1}^n\sum_{j=1}^M\kappa_j^{(i)}K_4\|\Theta^{(i)}_x [\eta_\ell]\|  \leq MK_5 \sum_{i=1}^n \sum_{j=1}^M \kappa_j^{(i)} e^{(\lambda_\ell+\varepsilon)i} \\
      &\leq K_6 \sum_{i=1}^n\mathrm{Vol}(\partial T\cdot B)\theta_{(i)_x}^{d-1} e^{(\lambda_\ell+\varepsilon)i},
    \end{split}
  \end{equation}
where we used (ii) of Lemma \ref{lem:packing} in the last inequality. By (\ref{eqn:Oseledets3}) and Oseledets theorem, for any $\delta>0$ there is a constant $K_7$ such that
\begin{equation}
  \label{eqn:Oseledets4}
  \theta_{(i)_x}^{d-1}\leq K_7 \mathrm{exp}\left(\left(\frac{1-d}{d}\lambda_1 + \delta \right)i\right)
  \end{equation}
for all $i>0$. Using this in (\ref{eqn:upBnd1}):
\begin{equation}
  \label{eqn:upBnd2}
    \begin{split}
      I_1 &\leq K_8 \sum_{i=1}^n\mathrm{Vol}(\partial T\cdot B)\mathrm{exp}\left(\left(\frac{1-d}{d}\lambda_1 + \delta \right)i+ (\lambda_\ell+\varepsilon)i\right) \\
      &\leq K_9 \mathrm{Vol}(\partial T\cdot B)\mathrm{exp}\left(\left(\frac{1-d}{d}\lambda_1 + \delta + \lambda_\ell+\varepsilon\right)n\right).
  \end{split}
\end{equation}
Using the bound for $\mathrm{Vol}(T\cdot B)$ given by Lemma \ref{lem:packing} and (\ref{eqn:Oseledets4}), we continue (\ref{eqn:upBnd2}):
\begin{equation}
  \label{eqn:upBnd3}
    \begin{split}
      I_1 &\leq  K_{10} \theta_{(n)_x}^{d-1}\mathrm{exp}\left(\left(\frac{1-d}{d}\lambda_1 + \delta + \lambda_\ell+\varepsilon\right)n\right) \\
      &\leq  K_{11} \mathrm{exp}\left(\left(\frac{d-1}{d}\lambda_1 + \delta \right)n\right) \mathrm{exp}\left(\left(\frac{1-d}{d}\lambda_1 + \delta + \lambda_\ell+\varepsilon\right)n\right) \\
      &=  K_{11}  \mathrm{exp}\left(\left( \lambda_\ell+\varepsilon+ 2\delta\right)n\right).
  \end{split}
\end{equation}
At this point we turn to $I_2$ in (\ref{eqn:twoInts}). The integral is over a neighborhood of the boundary. Thus  there exists a $C>0$ and a constant $K_{12}$ such that we have
$$I_2 \leq \left|\int_{\partial_C T\cdot B} \eta_\ell \right|\leq K_{12}C\mathrm{Vol}(\partial T\cdot B)\|\eta_\ell\|_\infty,$$
and, using (\ref{eqn:Oseledets3}), (\ref{eqn:returnTimes}), and (\ref{eqn:nTbounds}), for $\varepsilon>0$:
\begin{equation}
  \label{eqn:I2bound2}
  I_2 \leq K_{13}\mathrm{exp}\left( \left(\frac{d-1}{d}\lambda_1+\varepsilon\right)n\right).
  \end{equation}
Since by assumption $\eta_\ell$ represents a class in the rapidly expanding subspace, we have that $\lambda_\ell\geq\frac{d-1}{d}\lambda_1$ and therefore, comparing the bounds for $I_1$ and $I_2$, respectively in (\ref{eqn:upBnd3}) and (\ref{eqn:I2bound2}), the bound for $I_1$ dominates the bound for $I_2$, so there exists a $K_{14}>0$ such that
$$\left|\displaystyle \int_{T\cdot B}\eta_\ell \right| \leq K_{14} \mathrm{exp}\left(\left(\lambda_\ell + \varepsilon + 2\delta\right)n\right).$$
Finally, using (\ref{eqn:nTbounds}),
$$\frac{\log \left|\displaystyle \int_{T\cdot B}\eta_\ell \right| }{\log T} \leq \frac{\log K_{14} + (\lambda_\ell+\varepsilon+2\delta)n}{\log\theta_{min}^{m^-} +\log \theta_{(n)_x}}$$
which, through (\ref{eqn:returnTimes}), implies
$$\limsup_{T\rightarrow \infty}\frac{\log \left|\displaystyle \int_{T\cdot B}\eta_\ell \right| }{\log T} \leq \frac{d}{\lambda_1}(\lambda_\ell +\varepsilon+2\delta).$$
Since $\varepsilon, \delta$ are arbitrary, the result follows.
\end{proof}
\subsection{Lower Bound}
\begin{proposition}
  \label{prop:Low}
  Let $\mathcal{F} = \{\mathcal{F}_1,\dots, \mathcal{F}_N\}$ be a type H family, $\mu$ an ergodic, minimal $\sigma$-invariant probability measure on $\Sigma_N$, and $B\subset\mathbb{R}^d$ a compact subset with non-empty interior. For $\mu$-almost every $x$, every $\mathcal{T}\in \Omega_x$ and $\varepsilon>0$ there exists a compact subset $B_\varepsilon$ which is $\varepsilon$-close in the Hausdorff metric to $B$, a convergent sequence of vectors $\tau_k\in \mathbb{R}^d$ and a sequence $T_k\rightarrow \infty$ such that for any $\eta_\ell\in\Delta^d_\mathcal{T}$ representing a class in $E^+_\ell(x)$ of the type given by Proposition \ref{prop:generating} we have
  $$\limsup_{k\rightarrow \infty}\frac{\log \left| \displaystyle\int_{T_k(\tau_k+B^\varepsilon)} \eta_\ell  \right|}{\log T_k} \geq \nu_\ell,$$
    where $\nu_\ell$ is the $\ell^{th}$ normalized Lyapunov exponent of $\mu$.
\end{proposition}
\begin{proof}
  The set of points $x\in \Sigma_N$ for which $\mathcal{B}_x$ is minimal, satisfy Poincar\'e recurrence, Proposition \ref{prop:generating} and are Oseledets regular has full measure. Let $x$ be one such point, pick $\mathcal{T}\in\Omega_x$, and let $\bar{e}\in X_{\mathcal{B}_x}$ be such that $\Delta_x(\bar{e}) = \mathcal{T}$. Let $n_k\rightarrow \infty$ be a subsequence of times such that:
  \begin{itemize}
  \item $\mathcal{B}_x$ and $\mathcal{B}_{\sigma^{n_k}(x)}$ agree on levels indexed by $i$ with $|i|\leq k$;
    \item $\bar{e}\in \mathcal{B}_{x}$ and $\sigma^{n_k}(\bar{e})\in \mathcal{B}_{\sigma^{n_k}(x)}$ agree on all entries indexed by $i$ with $|i|\leq k$.
  \end{itemize}
That such subsequence exists follows from the fact that $x$ is Poincar\'e recurrent and that tiling spaces are compact.

Let $R_{\mathcal{F},1}$ denote the circumscribing radius of the prototiles $\{t_z\}$. That is, $R_{\mathcal{F},1}$ is the infimum of all $R$ such that for all $z\in\{1,\dots, M\}$ a ball of radius $R$ contains an isometric copy of the prototile $t_z$. By minimality, there exists a $R_{\mathcal{F},2}>0$ such that for any $\mathcal{T}\in\Omega_x$, $B_{R_{\mathcal{F},2}}(y)$ contains a copy of every collared tile for any $y\in\mathbb{R}^d$. Given $\varepsilon>0$ there exists an $T_\varepsilon\geq 0$ such that
\begin{itemize}
\item $T^{-1} \cdot \mathcal{O}_\mathcal{T}^-(T\cdot B)$ is $\varepsilon$-close to $B$ in the Hausdorff metric,
\item $ \mathcal{O}_\mathcal{T}^-(T \cdot B)$ contains a ball of radius $2\max \{\mathcal{R}_{\mathcal{F},1}, \mathcal{R}_{\mathcal{F},2}\}$
\end{itemize}
for all $T\geq T_\varepsilon$. Fix some $T_*\geq T_\varepsilon$ and define
\begin{equation}
  \label{eqn:Bepsilon}
  B_{\varepsilon}:= T^{-1}_*\cdot \mathcal{O}_\mathcal{T}^-(T_*\cdot B)\hspace{.7in}\mbox{ and }\hspace{.7in}\mathcal{P}_{x,\varepsilon, \mathcal{T}}:= T_*\cdot  B_{\varepsilon},
  \end{equation}
and note that $\mathcal{P}_{x,\varepsilon, \mathcal{T}}$ is a a patch of $\mathcal{T}$ which, by construction, contains a copy of every collared tile. Then there exists a $k_\varepsilon$, a finite set of paths $E_{\varepsilon,\mathcal{T}}\subset E_{\mathcal{V}_0,\mathcal{V}_{k_\varepsilon}}$ and a vector $\tau$ such that the patch $\mathcal{P}_{x,\varepsilon, \mathcal{T}}$ admits the decomposition
\begin{equation}
  \label{eqn:patch1}
  \mathcal{P}_{x,\varepsilon, \mathcal{T}} = \varphi_\tau\left(\bigcup_{\bar{e}'\in E_{\varepsilon,\mathcal{T}}} f_{\bar{e}|_{n_\varepsilon}}^{-1}\circ f_{\bar{e}'}(A_{s(\bar{e}')})\right),
  \end{equation}
where the vector $\tau$ is completely determined by the negative part of $\bar{e}$.  Note that by recurrence of $\Omega_x$ and the convergence $\sigma^{n_k}(\mathcal{T})\rightarrow \mathcal{T}$ for all $k$ large enough, there exists a vector $\tau_k\in\mathbb{R}^d$ such that $\tau_k+\mathcal{P}_{x,\varepsilon, \mathcal{T}}$ is a patch in $\Phi_x^{(n_k)}(\mathcal{T})$. By minimality/repetitivity there exists a compact set $\mathcal{K}\subset\mathbb{R}^d$ such that we can take all $\tau_k$ from $\mathcal{K}$, that is, $\tau_k\in\mathcal{K}$ for all $k$ large enough. By passing to a subsequence, we may assume that the sequence of vectors $\tau_k$ is convergent.

We now make an explicit decomposition of the patches $\tau_k+\mathcal{P}_{x,\varepsilon, \mathcal{T}} \subset\Phi_x^{(n_k)}(\mathcal{T})$. Since $x$ is Poincar\'e recurrent, for all $k$ large enough we have that $AP(\Omega_{\sigma^{n_k}(x)}) = AP(\Omega_x)$. Thus, for all $k$ large enough the set of collared tiles of $\Phi_x^{(n_k)}(\mathcal{T})$ is some fixed set $\{t^x_1,\dots, t_{\tilde{n}}^x\}$. We can decompose the patches $\tau_k + \mathcal{P}_{x,\varepsilon, \mathcal{T}}$ as
\begin{equation}
  \label{eqn:decomp2}
\tau_k + \mathcal{P}_{x,\varepsilon, \mathcal{T}} = \bigcup_{z=1}^{\tilde{n}}\bigcup_{y=1}^{s(z)} t_{z,y}^k
\end{equation}
as a patch in $\Phi_x^{(n_k)}(\mathcal{T})$, where, for each $k$, $\{t_{z,y}^k\}_{z,y}$ is a finite collection of tiles of $\Phi_x^{(n_k)}(\mathcal{T})$. This decomposition breaks down the set $\tau_k + \mathcal{P}_{x,\varepsilon, \mathcal{T}}$ as a finite union of $s(z)$ copies of each collared tile $t_z^x$. Since $\mathcal{P}_{x,\varepsilon, \mathcal{T}}$ contains a ball of radius $2\mathcal{R}_{\mathcal{F},2}$, it contains a copy of every prototile, so $s(z)\geq 1$ for all $z$.

  By Proposition \ref{prop:generating} we can choose the representative $\eta_\ell\in\Delta_\mathcal{T}^d$ of a particular form, namely a linear combination of forms of the form (\ref{eqn:generators}), where $\rho$ is a bump function whose support is small enough that each $\eta_j^i$ in (\ref{eqn:generators}) consists of a bump function supported entirely inside the tile being collared to give $P_j^i$.
For such a representative $\eta_\ell$ of a class in $E^+_\ell(x)$, we partition the set of indices $\{1,\dots ,\tilde{n}\}$ of the collared tiles into $I^+_\ell$ and $I^\circ_\ell$ as follows. An index $z\in I^+_\ell$ if and only if
  \begin{equation}
    \label{eqn:limsup1}
    \limsup_{k\rightarrow \infty}\frac{\left|\displaystyle\int_{t_{z,1}^k}A_{(n_k)_x}^* \eta_\ell   \right|}{\left\|\left[ A^*_{(n_k)_x}\eta_\ell\right]\right\|}>0,
  \end{equation}
  and is in $I^\circ_\ell$ if the limsup in (\ref{eqn:limsup1}) is zero. The tiles $t_{z,1}^k$ over which we integrate in (\ref{eqn:limsup1}) are the tiles from the decomposition (\ref{eqn:decomp2}).
  \begin{lemma}
    $I^+_\ell\neq \varnothing$.
  \end{lemma}

  \begin{proof}
    By Lemma \ref{lem:cocycleForms} we have
    $$\left\|\left[ A^*_{(n_k)_x}\eta_\ell\right]\right\| = \max_{t\in\Phi_x^{(n_k)}(\mathcal{T})}\left| \int_t A^*_{(n_k)_x}\eta_\ell \right|.$$
Thus $z\in I^+_\ell$ if and only if
\begin{equation}
\label{eqn:limsup2}
\limsup_{k\rightarrow\infty}\frac{\left|\displaystyle\int_{t_{z,1}^k}A_{(n_k)_x}^* \eta_\ell   \right|}{\displaystyle\max_{t\in\Phi_x^{(n_k)}(\mathcal{T})}\left| \displaystyle\int_t A^*_{(n_k)_x}\eta_\ell \right|}>0.
  \end{equation}
Note that the $\max$ over which the norm in Lemma \ref{lem:cocycleForms} is taken is not a $\sup$ because all the possible values of the integrals (under the hypotheses of Lemma \ref{lem:cocycleForms}) are given by integrating over all possible collared tiles of $\Phi_x^{(n_k)}(\mathcal{T})$. Thus the $\max$ is taken from all possible values given by the integrals over all possible collared tiles. Since for every $k$ the collection $\{t_{z,1}^k\}_z$ has at least one representative of each collared tile of $\Phi_x^{(n_k)}(\mathcal{T})$, then (\ref{eqn:limsup2}) holds for some $z$, so $I^+_\ell\neq 0$.
  \end{proof}
Note that by our choice of representative $\eta_\ell$ we have that for any $z\in \{1,\dots, \tilde{n}\}$ and any two $y,y'\in \{1,\dots, s(z)\}$ we have that
  $$    \limsup_{k\rightarrow \infty}\frac{\left|\displaystyle\int_{t_{z,y}^k}A_{(n_k)_x}^* \eta_\ell   \right|}{\left\|\left[ A^*_{(n_k)_x}\eta_\ell\right]\right\|} = \limsup_{k\rightarrow \infty}\frac{\left|\displaystyle\int_{t_{z,y'}^k}A_{(n_k)_x}^* \eta_\ell   \right|}{\left\|\left[ A^*_{(n_k)_x}\eta_\ell\right]\right\|}$$
since the integrals in the numerators only depend on the the collared type of a tile, which is the same for $t_{z,y}^k$ and $t_{z,y'}^k$.

Recall $s(z)$ denotes the number of copies of the collared tile $t_z^x$ found inside the set $\tau_k + \mathcal{P}_{x,\varepsilon, \mathcal{T}}$ defined in (\ref{eqn:decomp2}) (this is independent of $k$) and let $\kappa^\circ(B,x,\varepsilon)$ be the maximum of $\{s(z)\}$ for $z\in I^\circ_\ell$. Let $C$ be half of the smallest positive limsup in (\ref{eqn:limsup1}) for some $(z,y)\in I^+_\ell$. Since the limsup in (\ref{eqn:limsup1}) vanishes for $(z,y)\in I^\circ_\ell$, we have that for $k$ large enough,
  $$  \frac{\left|\displaystyle\int_{A_{(n_k)_x}t_{z,y}^k}\eta_\ell \right|}{\left\|\Theta_x^{(n_k)}[\eta_\ell]\right\|}\leq \frac{C|I^+_\ell|}{2|I^\circ_\ell|\kappa^\circ(B,x,\varepsilon)},\hspace{.2in}  \mbox{ whereas for $(z,y)\in I^+_\ell$ we have} \hspace{.35in} \frac{\left|\displaystyle\int_{A_{(n_k)_x}t_{z,y}^k}\eta_\ell \right|}{\left\|\Theta_x^{(n_k)}[\eta_\ell]\right\|}\geq C.$$
Now, by (\ref{eqn:decomp2}):
  \begin{equation}
    \label{eqn:lowBnd1}
\begin{split}
  \int_{A_{(n_k)_x}(\tau_k + \mathcal{P}_{x,\varepsilon, \mathcal{T}})} \eta_\ell &= \int_{\tau_k + \mathcal{P}_{x,\varepsilon, \mathcal{T}}} A_{(n_k)_x}^* \eta_\ell \\
  &= \sum_{z\in I^\circ_\ell} s(z) \int_{t_{z,1}^k}A_{(n_k)_x}^* \eta_\ell + \sum_{z\in I^+_\ell} s(z) \int_{t_{z,1}^k} A_{(n_k)_x}^* \eta_\ell.
  \end{split}
  \end{equation}
  Rearranging the terms in (\ref{eqn:lowBnd1}) and using the triangle inequality, for all $\ell$ large enough,
\begin{equation*}
    \label{eqn:lowBnd2}
\begin{split}
  \left| \int_{A_{(n_k)_x}(\tau_k + \mathcal{P}_{x,\varepsilon, \mathcal{T}})} \eta_\ell\right| &\geq  \left| \sum_{z\in I^+_\ell} s(z) \int_{t_{z,1}^k} A_{(n_k)_x}^* \eta_\ell \right| - \left| \sum_{z\in I^\circ_\ell} s(z) \int_{t_{z,1}^k}A_{(n_k)_x}^* \eta_\ell \right| \\
  &\geq C |I^+_\ell| \left\|\Theta^{(n_k)}_x [\eta_\ell]\right\| - \frac{C|I^+_\ell|}{2\kappa^\circ(B,x,\varepsilon)|I^\circ_\ell|} \kappa^\circ(B,x,\varepsilon)|I^\circ_\ell|\left\|\Theta^{(n_k)}_x [\eta_\ell]\right\| \\
  &= \frac{C|I^+_\ell|}{2}\left\|\Theta^{(n_k)}_x [\eta_\ell]\right\|.
  \end{split}
\end{equation*}
Note that by our choice of $B_\varepsilon$ in (\ref{eqn:Bepsilon}) we can write
$$A_{(n_k)_x}(\tau_k + \mathcal{P}_{x,\varepsilon, \mathcal{T}}) = A_{(n_k)_x}(T_*\cdot \mathrm{Id})(T_*^{-1}\tau_k + B_\varepsilon) = T_k\cdot (\bar{\tau}_k + B_\varepsilon),$$
where $\bar{\tau}_k = T^{-1}_*\tau_k$ and $T_k = T_*\theta_{(n_k)_x}^{-1}$. Finally, since $\eta_\ell$ represents a class in the Oseledets subspace $E^+_\ell(x)$, by (\ref{eqn:cocycles})-(\ref{eqn:returnTimes}),
\begin{equation*}
  \label{eqn:lowBound}
  \begin{split}
    \limsup_{k\rightarrow \infty}\frac{\log \left| \displaystyle\int_{T_k\cdot (\bar{\tau}_k + B_\varepsilon)} \eta_\ell  \right|}{\log T_k} &\geq \limsup_{k\rightarrow \infty}\frac{\log\left(\frac{C|I^+_\ell|}{2} \left\|\Theta_{x}^{(n_k)}[\eta_\ell]\right\|\right)}{\log T_{k}} \\
    &= \limsup_{k\rightarrow \infty}\frac{n_k}{\log T_{k}}\frac{\log \left\|\Theta_{x}^{(n_k)}[\eta_\ell]\right\|}{n_k}= d\frac{\lambda_\ell}{\lambda_1}=\nu_\ell.
  \end{split}
\end{equation*}

%
\end{proof}
\subsection{Proof of Theorem \ref{thm:main}}
  Let $\mathcal{F} = \{\mathcal{F}_1,\dots, \mathcal{F}_N\}$ be a type H family, $\mu$ a minimal $\sigma$-invariant ergodic probability measure on $\Sigma_N$. For an Oseledets regular $x\in\Sigma_N$ for the renormalization cocycle, and any $\mathcal{T}\in\Omega_x$ we pick a basis $\{[\eta_{1}],\dots, [\eta_{r_\mu}]\}$, where each class $[\eta_i]$ spans the Oseledets subspace $E_i(x)$ associated to the Lyapunov exponent $\lambda_i$ and is represented by the form $\eta_{i}\in \Delta_\mathcal{T}^d$ of the type given by Proposition \ref{prop:generating}, and we have ordered the spectrum such that $\lambda_1\geq \lambda_2\geq \cdots \geq \lambda_{r_\mu}$. Given $f\in C^\infty_{tlc}(\Omega_x)$, let $\bar{f} = i_\mathcal{T}f\in \Delta_\mathcal{T}^0$, which allows us to write it as
$$\bar{f} = \sum_{i=1}^{r_\mu}  \alpha_{i}(f) g_{i} + e_f,$$
since one initially has
$$\star \bar{f}  = \sum_{i=1}^{r_\mu}  \alpha_{i}(f) \eta_{i} + d\omega_f,$$
where $g_{i} = \star \eta_{i}$,  $e_f = d\omega_f/(\star 1)$ and $\alpha_{i}(f)$ are the components of the class $[\star \bar{f}]$ in the Oseledets space $E_{i}(x)$. Note that for any good Lipschitz domain $B$ one has
$$\left| \int_{T\cdot B}d\omega_f\right| = \left| \int_{\partial (T\cdot B)}\omega_f\right| \leq \|\omega_f\|_{\infty}\mathrm{Vol}_{d-1}(\partial T\cdot B)\leq K_f T^{d-1}$$
for some $K_f>0$. As a result, the contributions to the ergodic integral of $f$ will primarily come from the forms $\eta_i$ representing classes in the rapidly expanding subspace.

Define the distributions $\mathcal{D}_{i}$ as $\mathcal{D}_{i} = \alpha_{i}$ and denote by $\rho_\mu = \mathrm{dim}\,E^+_x$ the dimension of the rapidly expanding subspace, which is constant $\mu$-almost everywhere. Let $ f\in\Delta_\mathcal{T}^d $ and suppose $\alpha_i(f) = 0$ for all $i=1,\dots, j-1<\rho_\mu$ but $\alpha_j(f)\neq 0$. Then the decomposition of $\bar{f}$ reads
$$\bar{f} = \sum_{i=j}^{r_\mu}  \alpha_{i}(f) g_{i} + e_f.$$
For $B$ a good Lipschitz domain, Proposition \ref{prop:upBnd} gives
$$\limsup_{T\rightarrow \infty}\frac{\log \left|\displaystyle \int_{T\cdot B} f\circ\varphi_t(\mathcal{T})\, dt \right|}{\log T} = \limsup_{T\rightarrow \infty}\frac{\log \left|\displaystyle \int_{T\cdot B}\star \bar{f} \right|}{\log T} = \limsup_{T\rightarrow \infty}\frac{\log \left|\displaystyle \int_{T\cdot B}\eta_j \right|}{\log T} \leq \nu_j.$$
For any $\varepsilon>0$, let $B_\varepsilon$ be the good Lipschitz domain which is given by the proof of Proposition \ref{prop:Low}, along with the converging sequence of vectors $\{\tau_k\}$ and times $T_k\rightarrow \infty$. Proposition \ref{prop:Low} then gives
$$\limsup_{k\rightarrow \infty}\frac{\log \left|\displaystyle \int_{T_k\cdot(\tau_k+ B_\varepsilon)} f\circ\varphi_t(\mathcal{T})\, dt \right|}{\log T_k} = \limsup_{k\rightarrow \infty}\frac{\log \left|\displaystyle \int_{T_k\cdot(\tau_k+ B_\varepsilon)}\star \bar{f} \right|}{\log T_k} = \limsup_{k\rightarrow \infty}\frac{\log \left|\displaystyle \int_{T_k\cdot(\tau_k+ B_\varepsilon)}\eta_j \right|}{\log T_k} \geq \nu_j.$$
If $\alpha_i(f)=0$ for all $i=1,\dots, \rho_\mu$, then the boundary has the dominant effect. Indeed, in the decomposition (\ref{eqn:twoInts}) used for the upper bound, we showed how the growth of $I_1$ is controlled by the Lyapunov exponent in (\ref{eqn:upBnd3}) whereas the growth of $I_2$ is bounded by the growth of volume of $\partial T\cdot B$. Thus, it follows from (\ref{eqn:nTbounds}) and (\ref{eqn:I2bound2}) that
$$\limsup_{T\rightarrow \infty}\frac{\log \left|\displaystyle \int_{T\cdot B} f\circ\varphi_t(\mathcal{T})\, dt \right|}{\log T} = \limsup_{T\rightarrow \infty}\frac{\log \left|\displaystyle \int_{T\cdot B}\star \bar{f} \right|}{\log T} = \limsup_{T\rightarrow \infty}\frac{\log \left|\displaystyle \int_{T\cdot B}\eta_j \right|}{\log T} \leq d-1.$$

\bibliographystyle{amsalpha}
\bibliography{biblio}

\end{document}